\documentclass[11pt]{article}

\usepackage{enumerate}
\usepackage{amsmath}
\usepackage{amssymb,latexsym}
\usepackage{amsthm}
\usepackage{color}
\usepackage{graphicx}
\usepackage[notcite, notref]{showkeys}
\usepackage{amscd}
\usepackage{hyperref}

\title{A class of linear sets in $\PG(1,q^5)$}
\author{Maria Montanucci - Corrado Zanella}

\newcommand{\cB}{{\mathcal B}}

\newcommand{\Lb}{{\mathbb L}}

\newcommand{\cQ}{{\mathcal Q}}

\newcommand{\F}{{\mathbb F}}
\newcommand{\Fq}{{\F_q}}

\newcommand{\Fqq}{{\F_{q^5}}}
\newcommand{\Fqs}{{\F_{q^5}^*}}

\newcommand{\la}{\langle}
\newcommand{\ra}{\rangle}

\newcommand{\N}{\NNN_{q^5/q}}
\newcommand{\Tr}{\Tra_{q^5/q}}

\newtheorem{theorem}{Theorem}[section]
\newtheorem{lemma}[theorem]{Lemma}

\newtheorem{proposition}[theorem]{Proposition}

\DeclareMathOperator{\Tra}{Tr}
\DeclareMathOperator{\NNN}{N}
\DeclareMathOperator{\PG}{{PG}}

\DeclareMathOperator{\PGL}{{PGL}}
\DeclareMathOperator{\PGaL}{P\Gamma L}
\DeclareMathOperator{\GaL}{\Gamma L}

\DeclareMathOperator{\hht}{ht}

\theoremstyle{definition}

\newtheorem{remark}[theorem]{Remark}
\newtheorem{question}[theorem]{Question}
\begin{document}
\maketitle

\begin{abstract}
The maximum scattered linear sets in $\PG(1,q^n)$ have been completely classified
for $n\le 4$ \cite{CsZa2018,LaVdV2010}.
Here a wide class of linear sets  in $\PG(1,q^5)$ is studied
which depends on two parameters.
Conditions for the existence, in this class,
of possible new maximum scattered linear sets in $\PG(1,q^5)$
are exhibited.
\end{abstract}



\section{Introduction}

A point in $\PG(1,q^t)$ is the $\F_{q^t}$-span $\la {\bf v}\ra_{\F_{q^t}}$ of a nonzero vector
${\bf v}$ in a two-dimensional vector space, say $W$, over $\F_{q^t}$. If $U$ is a subspace over $\Fq$ of $W$, then
$L_U=\{\la {\bf v}\ra_{\F_{q^t}}\colon {\bf v}\in U\setminus\{ {\bf 0}\}\}$ denotes the associated
\emph{$\Fq$-linear set} (or simply \emph{linear set}) in $\PG(1,q^t)$. The \emph{rank} of such a linear set is $r=\dim_{\Fq}U$.
Any linear set in $\PG(1,q^t)$ of rank greater than $t$ coincides with the whole projective line.
The \emph{weight} of a point $P=\la {\bf v}\ra_{\F_{q^t}}$ of $L_U$ is 
$w_{L_U}(P)=\dim_{\Fq}(U\cap P)$.
If the rank and the size of $L_U$ are $r$ and $(q^r-1)/(q-1)$, respectively, then $L_U$ is \emph{scattered}.
Equivalently, $L_U$ is scattered if and only if all its points have weight one.
A scattered $\Fq$-linear set of rank $t$ in $\PG(1,q^t)$ is \emph{maximum scattered} (MSLS for short).
For any $\varphi\in\GaL(2,q^t)$ with related collineation $\tilde\varphi\in\PGaL(2,q^t)$ and any $\Fq$-linear set $L_U$, $L_{U^\varphi}=(L_U)^{\tilde\varphi}$.
As it was showed in \cite{CsZa2016a}, the converse is not true;
that is, there are examples of MSLSs $L_U=L_V\subseteq\PG(1,q^t)$ such that no
$\varphi\in\GaL(2,q^t)$ exists satisfying $U^\varphi=V$.
See also \cite{CsMaPo2018} on the problem of the $\GaL$-equivalence of the $\Fq$-subspaces
underlying to linear sets.

Up to our knowledge, only three types of MSLS in $\PG(1,q^5)$ are known:
\begin{itemize}
\item The  \emph{linear set of pseudoregulus type} 
$L_0=\{\langle(u,u^q)\rangle_{\mathbb{F}_{q^5}} \colon u\in\F_{q^5}^*\}$;
see \cite{CsZa2016} for a geometric description.
\item $L_1^\eta$ and $L_2^\eta$, where
\[
L_{s}^\eta=\{\langle (x,\eta x^{q^s}+x^{q^{5-s}}) \rangle_{\mathbb{F}_{q^5}} \ :\  x \in \mathbb{F}_{q^5}^*\},\ \N(\eta) \not\in \{0,1\}.
\]
They were constructed by Lunardon-Polverino \cite{LuPo2001} for $s=1$  and by Sheekey \cite{Sh2006} for $s=2$ 
(see also \cite{LuTrZh2015}.)
\end{itemize}
For any $\eta,\eta'$ with $\N(\eta),\N(\eta')^5\not\in\{0,1\}$, $L_1^\eta$ and $L_2^{\eta'}$ are not
$\PGaL(2,q^5)$-equivalent \cite[Theorem 5.5]{CsMaPo2018a}.
The aim of this paper is to find algebraic conditions for possible new examples
that on the other hand could also serve to prove their nonexistence.
Up to our knowledge, the problem of the classification of the MSLSs in $\PG(1,q^5)$ remains open.

In Sect.\ \ref{cforms}, a canonical form $L_{\alpha,\beta}$ is found for a wide class of 
linear sets in $\PG(1,q^5)$.
Based on the representation given in \cite[Theorems 1 and 2]{LuPo2004}, any linear set
$\Lb$ of rank five
in  $\PG(1,q^5)$ can be obtained as the projection
of a canonical subgeometry $\Sigma\cong\PG(4,q)$ 
 from a plane $\Lambda$ of $\PG(4,q^5)$ such that $\Lambda\cap\Sigma=\emptyset$.
Let $\sigma$ denote a generator of the collineation group fixing $\Sigma$ pointwise.
As a consequence of  \cite[Theorem 2.3]{CsZa2016}, assuming that the linear set $\Lb$ is maximum
scattered, it is a linear set of pseudoregulus type if and only if at least one of the intersections
$\Lambda\cap\Lambda^\sigma$ and $\Lambda\cap\Lambda^{\sigma^2}$ is not a point.
So it is assumed that $P=\Lambda\cap\Lambda^\sigma$ is a point.
Adding the assumption that the projective closure 
$\overline{P,P^\sigma,P^{\sigma^2},P^{\sigma^3},P^{\sigma^4}}$
is equal to $\PG(4,q^5)$ leads to the algebraic form
(\ref{e:canonicasi}) 
$L_{\alpha,\beta}=\{\la(x-\alpha x^{q^2},x^q-\beta x^{q^2})\ra_\Fqq\colon x\in\Fqs\}$ for $\Lb$.

Sects. 3 and 4 are based on the interpretation of algebraic equations in one unknown in $\Fqq$
as algebraic varieties in $\mathbb A^5(\Fq)$.
More precisely, taking a basis $\cB$ of $\Fqq$ over $\Fq$, from $f(x)=0$ a set of five equations
is obtained by equating to zero the coordinates of $f(x)$ with respect to $\cB$.

In Sect.\ 3 it is shown that asymptotically there are no MSLSs of type $L_{0,\beta}$.
This is consequence of a stronger result (Lemma \ref{lem1}), stating that for $q\ge223$ any element
of $\mathbb{F}_{q^5}^*$ is equal to $(uv^q-u^qv)/(u^{q^2}v-uv^{q^2})$ for 
some $u,v \in \mathbb{F}_{q^5}^*$ such that $\dim \langle u,v \rangle_{\mathbb{F}_q}=2$.
The proof is achieved by proving the existence of $\Fq$-rational points of the degree 5 hypersurface
(\ref{eq9}) in $\mathbb A^5(\Fq)$ not lying on a special hyperplane.
This is based on a recent bound by Slavov \cite{Sl2017}
for $\Fq$-rational points  on hypersurfaces (see Prop.\ \ref{prelvar}).
An exhaustive computer search allowed to extend such result also to $q\le17$.

Any MSLS of type $L_{\alpha,0}$ is of Lunardon-Polverino type.
If $\alpha^q=\beta^{q+1}$, then either $L_{\alpha,\beta}$ is of pseudoregulus type, or it has 
rank less than five (Prop.\ \ref{equivPSE}).
Motivated by this, in Sects.\ 4 and 5 MSLSs $L_{\alpha,\beta}$ are
dealt with under the assumption $\alpha\beta\neq0$.

Prop.\ \ref{nonMS} states that any MSLS of type $L_{\alpha,\beta}$
satisfies the condition $\alpha^q/\beta^{q+1}\in\Fq$.
This is a consequence of the existence of $\Fq$-rational points on a special
quartic curve $\cQ$ \eqref{quartica}.
In order to prove that, $\cQ$ is shown to be irreducible, allowing to apply
the Hasse-Weil bound.
No $L_{\alpha,\beta}$ with $\alpha\beta\neq0$ is of Lunardon-Polverino type (Lemma \ref{equivLP}).
A necessary and sufficient condition is proved for a MSLS $L_{\alpha,\beta}$ 
to be a Sheekey type linear set,
that for $q\le11$ is always satisfied (Theorem \ref{equivAll}).
The proof is based on the results by Csajb\'ok, Marino and Polverino
\cite[Theorem 5.4]{CsMaPo2018a}, implying that if a linear set $L_U$ is $\PGaL(2,q^5)$-equivalent
to a Sheekey's $L_2^\eta$, then $U$ is $\GaL(2,q^5)$-equivalent to the underlying $\Fq$-subspace
of a (possibly different) Sheekey linear set.

\section{Canonical forms}\label{cforms}

Let $\Sigma\cong\PG(4,q)$ be an $\Fq$-canonical subgeometry of $\PG(4,q^5)$;
that is, the set of all points of $\PG(4,q^5)$ having coordinates rational over $\Fq$
with respect to some projective reference system.
Furthermore, let $\sigma\in\PGaL(5,q^5)$
of order five fixing $\Sigma$ pointwise.
In this section $\Lb$ denotes a maximum scattered $\Fq$-linear set  in $\PG(1,q^5)$, not of pseudoregulus type.
By \cite{CsZa2016,LuPo2004}, $\Lb$ is the projection $p_\Lambda(\Sigma)$ 
with vertex a plane $\Lambda$ such that 
$\Lambda\cap\Sigma=\emptyset$,
and $\dim(\Lambda\cap\Lambda^\tau)=0$ for any generator $\tau$ of $\la\sigma\ra$.

The \emph{standard subgeometry} $\Sigma$ is the set of all points of type
\[
P_u=\la(u,u^q,u^{q^2},u^{q^3},u^{q^4})\ra_\Fqq,\ u\in\Fqs.
\]
and $P_u=P_v$ if and only if $u/v\in\Fq$. A possible choice for $\sigma$ is
\[
\sigma: \la(X_0,X_1,X_2,X_3,X_4)\ra_\Fqq\mapsto\la(X_4^q,X_0^q,X_1^q,X_2^q,X_3^q)\ra_\Fqq.
\]
The \emph{height} of a point $P$ with respect $\Sigma$, denoted by  $\hht P$, is the projective dimension of the
$\sigma$-cyclic subspace
$\overline{P,P^\sigma,P^{\sigma^2},P^{\sigma^3},P^{\sigma^4}}$
(\footnote{$\overline S$ denotes the projective closure of $S$.}).
Note that $\hht(\Lambda\cap\Lambda^\sigma)=\hht(\Lambda\cap\Lambda^{\sigma^4})$ and
$\hht(\Lambda\cap\Lambda^{\sigma^2})=\hht(\Lambda\cap\Lambda^{\sigma^3})$.

As usual, if $f(x)=\sum_{i=0}^4a_ix^{q^i}$ is a $q$-polynomial, then
\[ L_f=\{\la(x,f(x)\ra_{\Fqq}\colon x\in\Fqs\} \]
denotes the related linear set.

\begin{proposition}
There exists a $q$-polynomial $f(x)=\sum_{i=1}^4a_ix^{q^i}$ with $a_4=1$, such that $\Lb$ is
projectively equivalent to $L_f$,
or $\Lb$ is projectively equivalent to $L_g$ where $g(x)=ax^{q^2}+x^{q^3}$, $a\in\Fqs$,
$\N(a)\neq0,1$.
\end{proposition}
\begin{proof}
Up to projective equivalence, $\Lb=L_h$ with $h=\sum_{i=1}^4a_ix^{q^i}$ may be assumed.
If $a_4\neq0$ a further projectivity leads to $a_4=1$.
If $a_1\neq0$, then $\Lb=L_{\hat h}$ where $\hat h=\sum_{i=1}^4a_i^{q^{5-i}}x^{q^{5-i}}$
\cite[Lemma 2.6]{BaGiMaPo2018}, \cite[Lemma 3.1]{CsMaPo2018}, leading
once again to the desired form.
Finally, if $a_1=a_4=0$, then $a_2a_3\neq0$ since otherwise $\Lb$ would be of pseudoregulus type.
In this case $\N(a)\neq1$ is a necessary
and sufficient condition for the linear set to be scattered \cite[Cor.\ 3.7]{BaZhou}. 
\end{proof}

In the following, $O_0=\la(1,0,0,0,0)\ra_\Fqq$, $O_1=\la(0,1,0,0,0)\ra_\Fqq$, and so on.
\begin{proposition} \label{height Sh}
Let $g(x)=ax^{q^2}+x^{q^3}$, $a\in\Fqs$.
Then $L_g$ is the projection of the standard subgeometry from the vertex 
\[
\Lambda=\overline{O_1,O_4,\la(0,0,1,-a,0)\ra_\Fqq}.
\]
The intersection $\Lambda\cap\Lambda^{\sigma^i}$ is a point for any $i=1,2,3,4$.
Furthermore, $\Lambda\cap\Lambda^\sigma$ has height four if and only if $\N(a)^2-\N(a)+1\neq0$,
whereas $\Lambda\cap\Lambda^{\sigma^2}=O_1$ has height four for any $a\in\Fqs$.
\end{proposition}
\begin{proof}
As regards the first assertion, just take into consideration the following singular matrix:
\begin{equation}\label{e:vertice}
\begin{pmatrix} 0&0&0&0&1\\ 0&0&1&-a&0\\ 0&1&0&0&0\\ u&u^q&u^{q^2}&u^{q^3}&u^{q^4}\\
u&0&0&au^{q^2}+u^{q^3}&0 \end{pmatrix}.
\end{equation}
Straightforward computations give $\dim(\overline{\Lambda\cup\Lambda^\sigma})=
\dim(\overline{\Lambda\cup\Lambda^{\sigma^2}})=4$.
The intersection $\Lambda\cap \Lambda^\sigma$ is the point $\la(0,0,1,-a,a^{q+1})\ra_\Fqq$,
and
\[
\det\begin{pmatrix} 0&0&1&-a&a^{q+1}\\ a^{q^2+q}&0&0&1&-a^q\\ -a^{q^2}&a^{q^3+q^2}&0&0&1\\
1&-a^{q^3}&a^{q^4+q^3}&0&0\\ 0&1&-a^{q^4}&a^{q^4+1}&0 \end{pmatrix}=\N(a)^2-\N(a)+1.\qedhere
\]
\end{proof}
The proof of the following is similar to Prop.\ \ref{height Sh}:
\begin{proposition}
The Lunardon-Polverino linear set $L_f$ with $f=ax^q+x^{q^4}$, $\N(a)\neq0,1$, 
is the projection of the standard subgeometry from the vertex 
\[
\Lambda=\overline{O_2,O_3,\la(0,1,0,0,-a)\ra_\Fqq}.
\]
The point $\Lambda\cap\Lambda^\sigma=O_3$ has height four, whereas $\Lambda\cap\Lambda^{\sigma^2}$ has height four
if and only if $\N(a)^2-\N(a)+1\neq0$.
\end{proposition}

\begin{proposition} \label{height4}
Assume $\hht(\Lambda\cap\Lambda^\sigma)=4$.
Then, up to projectivities, 
\begin{equation}\label{e:canonicasi}
\Lb=L_{\alpha,\beta}=\{\la(x-\alpha x^{q^2},x^q-\beta x^{q^2})\ra_\Fqq\colon x\in\Fqs\}
\end{equation}
for some $\alpha,\beta\in\Fqq$ satisfying $\alpha^q\neq\beta^{q+1}$.
\end{proposition}
\begin{proof}
Since the setwise stabilizer $\PGL(5,q^5)_{\{\Sigma\}}$ acts transitively on the points of $\PG(4,q^5)$ of height four,
  \cite[Proposition 3.1]{BoPo2005}, it may be assumed that $O_4=\Lambda\cap\Lambda^\sigma$.
  This in turn implies $O_3\in\Lambda$, and
\begin{equation}\label{e:zero}
  \Lambda=\overline{\la(a,b,c,0,0)\ra_{\Fqq},O_3,O_4}
\end{equation}
  for some $a,b,c\in\Fqq$, not all zero.
The hyperplane coordinates of the span of $\Lambda$ and $P_u$ are
\[ [cu^q-bu^{q^2}:-cu+au^{q^2}:bu-au^q:0:0]. \]
So, for $c=0$ the linear set $\Lb$ is projectively equivalent to
\[
\{\la(x^{q^2},bx-ax^{q})\ra_\Fqq\colon x\in\Fqs\}=\{\la(x,bx^{q^3}-ax^{q^4})\ra_\Fqq\colon x\in\Fqs\},
\]
and by  \cite[Lemma 2.6]{BaGiMaPo2018}, \cite[Lemma 3.1]{CsMaPo2018} this can be expressed in the form 
$L_f$ where $f=d x^q+e x^{q^2}$; more precisely, $d=-a^q$ and $e=b^{q^2}$. 
Since $L$ is not of pseudoregulus type, $de\neq0$. In this case $\Lb$ is projectively equivalent
to $L_{0,-ed^{-1}}$.
If $c\neq0$, then $c=1$ may be assumed.
Let $f_1(u)=u^q-bu^{q^2}$, $f_2(u)=-u+au^{q^2}$, $f_3(u)=bu-au^q$.
Clearly $f_3=-af_1-bf_2$.
So, taking into account that $\Lb$ is scattered, the pairs $(f_1(u),f_2(u))$ and $(f_1(v),f_2(v))$
are $\Fq$-linearly dependent if and only if $u$ and $v$ are.
Therefore $f_1(u)$ and $f_2(u)$ can be
chosen as homogeneous coordinates of the points of $\Lb$.

If the intersection $\Lambda\cap\Lambda^\sigma$ is not a point then $\Lb$ is a linear set of pseudoregulus
type, a contradiction.
Furthermore, direct computations show that $\Lambda\cap\Lambda^\sigma$ is a point if,
and only if, $b^{q+1}-ca^q\neq0$.
This implies  $\alpha^q\neq\beta^{q+1}$.
\end{proof}
\begin{proposition}\label{equivPSE}
\begin{enumerate}[(i)]
\item The linear set $L_{\alpha,\beta}$ has rank less than five if and only if
\begin{equation}\label{eqalphabeta}
\alpha^q=\beta^{q+1}\quad\mbox{ and }\quad \N(\alpha)=\N(\beta)=1.
\end{equation}
\item If $\alpha^q\neq\beta^{q+1}$, then $L_{\alpha,\beta}$ is not of pseudoregulus type.
\item If $\alpha^q=\beta^{q+1}$ and $(\N(\alpha),\N(\beta))\neq(1,1)$, then $L_{\alpha,\beta}$ is 
of pseudoregulus type.
\end{enumerate}
\end{proposition}
\begin{proof}
Note that $x-\alpha x^{q^2}$ has non-trivial zeros if and only if $\N(\alpha)=1$ and $x^q-\beta x^{q^2}$ has non-trivial zeros if and only if $\N(\beta)=1$.
Also, $L_{\alpha,\beta}$ is of rank less than $5$ if and only if there is a common non-trivial root of the defining polynomials, that is, $x \in \mathbb{F}_{q^5}^*$ such that $x^{(q+1)(1-q)}=\alpha$ and $x^{(1-q)q}=\beta$. This is equivalent to (\ref{eqalphabeta}).

Since for $\alpha^q\neq\beta^{q+1}$ both $\Lambda\cap\Lambda^\sigma$ and 
$\Lambda\cap\Lambda^{\sigma^2}$ are points, no linear set of type $L_{\alpha,\beta}$
satisfying such inequality is of pseudoregulus type, whereas, as mentioned in proof of
Prop.\ \ref{height4}, if $\alpha^q=\beta^{q+1}$, then $\Lambda\cap\Lambda^\sigma$ is a line,
so $L_{\alpha,\beta}$ is of pseudoregulus type.
\end{proof}

\noindent\textbf{Remarks.}\\
1) By Proposition \ref{prop2} and the subsequent remark, for $\beta\neq0$ no $L_{0,\beta}$ 
is scattered for 
$q \geq 223$ or $q\le17$.\\
2) For $\beta=0$ and $\N(\alpha)\neq-1$, (\ref{e:canonicasi}) defines a linear set of Lunardon-Polverino type.
As a matter of fact take $y=x^q$, then up to projective equivalence $L_{\alpha,\beta}$ is
\[ \{\la(y,-\alpha y+y^{q^4})\ra_\Fqq\colon y\in\Fqs\} \]
which is maximum scattered if and only if $\N(-\alpha)\neq1$ \cite{LuPo2001}.\\
3) Similarly to Proposition \ref{height4}, if $\hht(\Lambda\cap\Lambda^{\sigma^2})=4$
and $\Lb$ is not of Sheekey type,
then $\Lb$ is projectively equivalent to
\[ \{\la(x-\alpha x^{q^3},x^q-\beta x^{q^3})\ra_\Fqq\colon x\in\Fqs\}
\]
for some $\alpha,\beta\in\Fqq$ not both zero.

\section{On some binomial linear sets}

\begin{lemma} \label{lem1} 
Let $q \geq 223$. Then any $b \in \mathbb{F}_{q^5}^*$ can be written as
\begin{equation} \label{eq1}
b=\frac{uv^q-u^qv}{u^{q^2}v-uv^{q^2}}
\end{equation}
for some $u,v \in \mathbb{F}_{q^5}^*$ such that $\dim \langle u,v \rangle_{\mathbb{F}_q}=2$.
\end{lemma}

We will use the following preliminary result.

\begin{lemma}{\rm{\cite[Corollary 7]{Sl2017}}} \label{prelvar}
Let $G \in \mathbb{F}_q [x_1,\ldots,x_n]$ be an absolutely irreducible polynomial of degree $d$, and let $H \in \mathbb{F}_q [x_1,\ldots,x_n]$ be a polynomial of degree $e$, not divisible by $G$. Then there exists a nonsingular zero of $G$ over $\mathbb{F}_q$, which is not a zero of $H$, provided that
\begin{equation} \label{eqs}
q > \frac{1}{4}\bigg((d-1)((d-2)+\sqrt{(d-1)^2(d-2)^2+4(d^2+de+10)}\bigg)^2.
\end{equation}
\end{lemma}

Now we can proceed with the proof of Lemma \ref{lem1}.

\begin{proof}
First note that the right hand side of (\ref{eq1}) only makes sense when $\dim \langle u,v \rangle_{\mathbb{F}_q}=2$.

Let $b$ be an arbitrary element of $\mathbb{F}_{q^5}^*$. Clearly, $$(u^{q^2}v-uv^{q^2})b=uv^q-u^qv$$
holds for $u,v \in \mathbb{F}_{q^5}^*$ if and only if
$$v^{q^2+1}\bigg( \bigg( \frac{u}{v}\bigg)^{q^2} - \bigg( \frac{u}{v}\bigg) \bigg)b=-v^{q+1} \bigg( \bigg( \frac{u}{v}\bigg)^q-\bigg( \frac{u}{v}\bigg)\bigg),$$
which is equivalent to

\begin{equation} \label{eq2}
\bigg( \bigg( \frac{u}{v}\bigg)^{q^2} - \bigg( \frac{u}{v}\bigg) \bigg)b=-\frac{1}{v^{q^2-q}} \bigg( \bigg( \frac{u}{v}\bigg)^q-\bigg( \frac{u}{v}\bigg)\bigg),
\end{equation}
since $v \ne 0$. Let $x:=u/v$ and $y=1/v^q$. Then   \eqref{eq2} reads,

\begin{equation} \label{eq3}
y^{q-1}=-b \bigg( \frac{x^{q^2}-x}{x^q-x}\bigg).
\end{equation}

Note that if we can find a couple $(x,y)$ where $x \in \mathbb{F}_{q^5} \setminus \mathbb{F}_q$ and $y \in \mathbb{F}_{q^5}^*$ such that   \eqref{eq3} is satisfied, then we can find a couple $(u,v) \in \mathbb{F}_{q^5}^* \times \mathbb{F}_{q^5}^*$ satisfying   \eqref{eq2} simply defining $v=\nu$, where $\nu^q=1/y$ and $u=vx$. 

Given $x \in \mathbb{F}_{q^5} \setminus \mathbb{F}_q$, there exists $y \in \mathbb{F}_{q^5}^*$ such that   \eqref{eq3} is satisfied if and only if 

\begin{equation} \label{eq4}
\bigg(-b \bigg( \frac{x^{q^2}-x}{x^q-x}\bigg)\bigg)^m=-b^m \bigg( \frac{x^{q^2}-x}{x^q-x}\bigg)^m=1,
\end{equation}
where $m=(q^5-1)/(q-1)$. In fact if $y \in \mathbb{F}_{q^5}^*$ exists then it is sufficient to use that $(y^{q-1})^m=y^{q^5-1}=1$ to note that   \eqref{eq4} is satisfied. Conversely, if   \eqref{eq4} is satisfied, then $-b(x^{q^2}-x)/(x^q-x)$ is a $(q-1)$-th power in $\mathbb{F}_{q^5}$ and hence it is sufficient to define $y$ to be an arbitrary $(q-1)$-th root of $-b(x^{q^2}-x)/(x^q-x)$.

Hence our aim is to show that for any $a\in \mathbb{F}_{q}^*$, there exists $x \in \mathbb{F}_{q^5} \setminus \mathbb{F}_q$ such that

\begin{equation} \label{eq5}
(x^{q^2}-x)^m=a(x^q-x)^m,
\end{equation}
so that defining $a:=-b^m$ the claim will follow.

A geometrical interpretation of \eqref{eq5} as the set of $\mathbb{F}_q$-rational points of an algebraic variety in $\mathbb{A}^5(\mathbb{F}_q)$ can be given as follows. 

From \cite[Theorem 2.35]{LiNi1997} we know that $\mathbb{F}_{q^5}$ admits a normal basis over $\mathbb{F}_q$, that is a basis of type $\{\gamma,\gamma^q,\gamma^{q^2},\gamma^{q^3},\gamma^{q^4}\}$ for some $\gamma \in \mathbb{F}_{q^5} \setminus \mathbb{F}_q$. So every solution $x$ of \eqref{eq5} can be written as $x=\sum_{i=0}^{4} x_i \gamma^{q^i}$ where $x_i \in \mathbb{F}_q$ for every $i=0,\ldots 4$. By applying the identification $\mathbb{F}_{q^5} \cong \mathbb{F}_q^5$ the $q$ elements of $\mathbb{F}_q$ in $\mathbb{F}_{q^5}$ can be identified with the elements of type $x=\sum_{i=0}^{4} \xi \gamma^{q^i}$ where $\xi \in \mathbb{F}_q$ as 
$\Tr(\gamma)=\gamma+\gamma^q+\gamma^{q^2}+\gamma^{q^3}+\gamma^{q^4} \in \mathbb{F}^*_q$; while \eqref{eq5} can be rewritten as a system of $5$ equations in $5$ variables of type

\begin{equation} \label{eq8}
\mathcal{V}: 
\begin{cases}
C_0(x_0,\ldots,x_4)=a',\\
C_1(x_0,\ldots,x_4)=a', \\
C_2(x_0,\ldots,x_4)=a', \\
C_3(x_0,\ldots,x_4)=a',\\
C_4(x_0,\ldots,x_4)=a'
\end{cases}
\end{equation}
where $a'=a(\Tr(\gamma))^{-1}$. 
Indeed the algebraic variety $\mathcal{V} \subseteq \mathbb{A}^4(\mathbb{F}_q)$ is obtained by forcing each coefficient $C_i(x_0,\ldots,x_4)$ of $\gamma^{q^i}$ in $(x^{q^2}-x)^m/(x^q-x)^m=((x^q-x)^{q-1}+1)^m$ to be equal to $a'$ for $i=0,\ldots,4$.

We apply the following change of variables in $\mathbb{F}_{q^5}$ (whose matrix is a so-called
Moore matrix and is nonsingular):

$$\begin{cases} A=x_0 \gamma+x_1 \gamma^q + x_2 \gamma^{q^2}+x_3 \gamma^{q^3} + x_4 \gamma^{q^4}, \\ B=x_4 \gamma+x_0 \gamma^q + x_1 \gamma^{q^2}+x_2 \gamma^{q^3} + x_3 \gamma^{q^4}, \\ C=x_3 \gamma+x_4 \gamma^q + x_0 \gamma^{q^2}+x_1 \gamma^{q^3} + x_2 \gamma^{q^4}, \\ D=x_2 \gamma+x_3 \gamma^q + x_4 \gamma^{q^2}+x_0 \gamma^{q^3} + x_1 \gamma^{q^4}, \\ E=x_1 \gamma+x_2 \gamma^q + x_3 \gamma^{q^2}+x_4 \gamma^{q^3} + x_0 \gamma^{q^4},\end{cases}$$
that is $(A,B,C,D,E)=(x,x^q,x^{q^2},x^{q^3},x^{q^4})$. 
In these new variables, recalling that $m=q^4+q^3+q^2+q+1$, \eqref{eq5} reads,

\begin{equation} \label{eq9}
\mathcal{H}: (C-A)(D-B)(E-C)(A-D)(B-E)-a(B-A)(C-B)(D-C)(E-D)(A-E)=0, 
\end{equation}
which is a hypersurface in $\mathbb{A}^5(\mathbb{F}_{q^5})$.
We showed that the change of variables implies that the algebraic variety $\mathcal{V} \subseteq \mathbb{A}^5(\mathbb{F}_q)$ is birationally isomorphic to the hypersurface $\mathcal{H}$ over $\mathbb{F}_{q^5}$.
Since the dimension of a variety is a birational invariant, also $\mathcal{V}$ is a hypersurface of degree $5$ in $\mathbb{A}^5(\mathbb{F}_{q})$, that is $C_i=C_j$ for $i,j=0,\ldots,4$.

Also, for the same reason we can show that $\mathcal{H}$ is absolutely irreducible
to prove the absolute irreducibility of $\mathcal{V}$.

To ensure the existence of at least one point of \eqref{eq8}, we will use the following strategy.

\begin{itemize}

\item We prove that $\mathcal{H}$ is absolutely irreducible, so that $\mathcal{V} \subseteq \mathbb{A}^5(\mathbb{F}_q)$ is an absolutely irreducible hypersurface of degree $5$.

\item  We apply Lemma \ref{prelvar} with respect to the hyperplane $H(x_0,\ldots,x_4)=x_0-x_1=0$ to ensure the existence of a point $P=(p_0,p_1,p_2,p_3,p_4) \in \mathcal{V}$ with $p_0 \ne p_1$. Recalling that the elements in $\mathbb{F}_q$ are identified with the vectors in $\mathbb{F}_q^5$ of type $(a,a,a,a,a)$ with $a \in \mathbb{F}_q$ this implies the existence of a solution $x \in \mathbb{F}_{q^5} \setminus \mathbb{F}_q$ of \eqref{eq5}. 
\end{itemize}

Since the degree of $\mathcal{H}$ is five either $\mathcal{H}$ is absolutely irreducible, or it has 
a linear component (hyperplane) or it splits in an absolutely irreducible cubic and an absolutely 
irreducible quadric. 
We divide the proof in two steps accordingly.

\begin{itemize}
\item \textbf{Step 1: $\mathcal{H}$ has no linear component.}
Let $t: a_1 A+b_1 B+c_1 C+ d_1 D+ e_1 E+ f_1=0$ be a linear component of $\mathcal{H}$.

If $a_1 \ne 0$ then $A=(b_1 B+c_1 C+ d_1 D+ e_1 E +f_1)/a_1$. Substituting in $\mathcal{H}$ and considering the evaluation at $(A,B,C,0,0)$ we get that $b_1=0$ and since $a_1 \ne 0$ also $c_1=f_1=0$. Considering then the evaluation at $(A,B,C,D,0)$ since $a \ne 0$ we get that $d_1=0$ yielding $B^2C^2Da_1^2+BC^2D^2a_1^2=0$, so that $a_1=0$; a contradiction.

Assume that $a_1=0$ but $b_1 \ne 0$. Then $B=(c_1C+d_1 D+e_1 E+f_1)/b_1$ and substituting in $\mathcal{H}$ and considering the valuation in $(0,B,0,D,E)$ we get $d_1=e_1=f_1=0$. Evaluating then in $(0,B,C,D,E)$ we get that $c_1=b_1=0$ which is a contradiction.

Assume that $a_1=b_1=0$ and $c_1 \ne 0$. Then $C=(d_1 D + e_1 E + f_1)/c_1$. Considering the evaluation of $\mathcal{H}$ is $(0,0,C,D,E)$ we get that $d_1=f_1=0$. Now the evaluation at $(0,B,C,D,E)$ gives $c_1=0$ and hence a contradiction.

Assume that $a_1=b_1=c_1=0$ but $d_1 \ne 0$. Then substituting $D=(e_1 E + f_1)/d_1$ in $\mathcal{H}$ gives $A^2 B^2 c d_1+A^2 B^2 E d_1^2+\ldots=0$, so that $d_1=0$; a contradiction.

Finally $a_1=b_1=c_1=d_1=0$ and $e_1 \ne 0$ otherwise $t$ would be a constant. From $\mathcal{H}$ we get that $A^2B^2Ce_1^2+\ldots+CD^2f_1^2=0$ so that $e_1=f_1=0$, which is not possible.

\item \textbf{Step 2: $\mathcal{H}$ does not split as the product of an absolutely irreducible cubic and an absolutely irreducible quadric.}
Assume by contradiction that the quadric
$$\mathcal{C}: a_2 A^2+b_2 B^2+c_2 C^2+d_2 E^2+ab_{1,1} AB+ac_{1,1} AC+ad_{1,1} AD$$
$$+ae_{1,1} AE+bc_{1,1}BC+bd_{1,1} BD+be_{1,1} BE+cd_{1,1} CD + c_1 e_1 CE+d_1 e_1 DE$$
$$+a_1 A+b_1 B+c_1 C+d_1D + e_1 E+f_0=0$$
is an absolutely irreducible component of $\mathcal{H}$. Evaluating the resultant of $\mathcal{H}$ and $\mathcal{C}$ with respect to $A$ in $(A,B,C,0,0)$ we get that $B^8C^2b_2^2+\ldots+B^4C^2f_0^2=0$ and hence $b_2=f_0=b_1=0$. Substituting the obtained values of the parameters we also get that $c_1(a_1+c_1)=0$ and thus either $c_1=0$ or $a_1=-c_1$. 

Assume first that $c_1=0$. Considering the valuation of the result in $(A,0,C,D,0)$ we get that $C^8D^2a^2c_2^2+\ldots+C^2D^8a^2d_2^2+\ldots+C^2D^6a^2d_1^2=0$ so that $c_2=d_2=d_1=0$.
Considering now the valuation at $(A,0,0,D,E)$ one has $D^2E^6 e_1^2+\ldots + D^2 E^8 e_2^2=0$, so that $e_1=e_2=0$. Considering the valuation at $(A,B,0,0,E)$ we get $b_1 e_1=0$ while from the valuation at $(A,0,D,D,E)$ we get $a_1=0$. Since this implies that $a_1=0=-c_1$, this condition can be assumed from the beginning.

Suppose hence that $a_1=-c_1$. 
From the valuations of the resultant in $(A,B,B,0,E)$, $(A,B,0,0,E)$, $(A,0,C,D,E)$ and $(A,0,C,D,C)$ we get that $e_1=e_2=0$, $be_{1,1}=0$, $c_2=0$ and $d_2=0$ respectively. Analyzing the structure of the valuation in $(A,B,C,D,0)$ we get that all the quadratic terms in $\mathcal{C}$ must be equal to zero and hence a contradiction.

This method can fail only if all the coefficients of terms involving $A$ in $\mathcal{C}$ are equal to zero. However, similar contradictions can be obtained assuming that all the coefficients of $A$ are equal to zero but at least one coefficient in the remaining variables is not equal to zero.
\end{itemize}

This shows that $\mathcal{H}$ (and hence $\mathcal{V}$) is absolutely irreducible. 

From Lemma \ref{prelvar} applied with respect to the hyperplane $H(x_0,\ldots,x_4)=x_0-x_1$ we get that if 
\[
q>\left\lfloor\frac{1}{4}\bigg((5-1)((5-2)+\sqrt{(5-1)^2(5-2)^2+4(25+5 \cdot 1+10)}\bigg)^2
\right\rfloor=216
\]
then $\mathcal{V}$ has at least an $\mathbb{F}_{q}$-rational point $P$ which does not correspond to a solution of \eqref{eq5} in $\mathbb{F}_{q}$. Since in our hypothesis $q \geq 223$ the claim follows.
\end{proof}

\

\begin{proposition} \label{prop2}
Let $q \geq 223$ and let $f(x)=x^q+bx^{q^2}$ for some $b \in \mathbb{F}_{q^5}^*$. Then $L_f=\{\langle (x,f(x))\rangle_{\mathbb{F}_{q^5}} \ : \ x \in \mathbb{F}_{q^5}^*\}$ is not a maximum scattered linear set of $\PG(1,q^5)$.
\end{proposition}

\begin{proof}
It is enough to show that there exists $m \in \mathbb{F}_{q^5}$ such that $h_m(x):=mx+x^q+bx^{q^2}$ has $q^2$ roots in $\mathbb{F}_{q^5}$. From Lemma \ref{lem1}, there exist $u,v \in \mathbb{F}_{q^5}^*$ such that \eqref{eq1} is satisfied and $\dim \langle u,v \rangle_{\mathbb{F}_q}=2$. Put $m=(u^qv^{q^2}-u^{q^2}v^q)/(u^{q^2}v-uv^{q^2})$. Then by direct checking $h_m(u)=h_m(v)=0$. 
\end{proof}

\begin{remark}
Prop.\ \ref{prop2} has been extended by an exhaustive computer search using GAP also to any 
$q\le17$.
\end{remark}  

\section{The linear sets $L_{\alpha,\beta}$}

Let $L_{\alpha,\beta}$ denote the linear set defined in (\ref{e:canonicasi}).
Motivated by Props.\ \ref{equivPSE} and \ref{prop2} and Rem.\ 2) at the end of Sect.\ 2,
we will always assume $\alpha^q\neq\beta^{q+1}$ and $\alpha\beta\neq0$.
Since the point $\la(0,1)\ra_{\mathbb{F}_{q^5}}$ has weight less or equal to one,
$L_{\alpha,\beta}$ is maximum scattered if and only if there is no $m \in \mathbb{F}_{q^5}$ such that 
\begin{equation} \label{eqh}
h_m(x):=m(x-\alpha x^{q^2})+(x^q-\beta x^{q^2})=mx+x^q-x^{q^2}(\beta +m \alpha)
\end{equation}
has $q^2$ roots in $\mathbb{F}_{q^5}$, that is, if and only if 
there is no $m \in \mathbb{F}_{q^5}$ such that $h_m(x)$ has a two-dimensional  kernel.

Using this fact, we prove the following characterization of maximum scattered $\mathbb{F}_q$-linear sets of type $L_{\alpha,\beta}$. 
It follows as a direct application of \cite[Theorem 3.3 and Section 3.3]{CsMaPoZu2018}.

\begin{lemma} \label{char}
Let $\alpha,\beta \in \mathbb{F}_{q^5}$ with $(\alpha,\beta) \ne (0,0)$. Then $L_{\alpha,\beta}$ is maximum scattered if and only if there is no $\lambda \in \mathbb{F}_{q^5}$ such that
(\footnote{In order to simplify the notation, from now on we write $N(-)$ instead of $\N(-)$.})
\begin{equation} \label{char1}
\begin{cases}
N(\lambda)=-1, \\ 
\lambda^q \beta^{q^3+q+1}+\beta^{q^3}(1-\lambda \alpha)^{q+1}-\lambda^{q^2+q+1} (1-\lambda \alpha)^{q^3} \beta^{q+1}=0.
\end{cases}
\end{equation}
\end{lemma}

\begin{proof}
As recalled, $L_{\alpha,\beta}$ is maximum scattered if and only if there is no 
$m \in \mathbb{F}_{q^5}$ such that $h_m(x)$ has maximum kernel. 
We note that both $m \ne 0$ and $\beta+m \alpha \ne 0$ can be assumed. 
Indeed $h_0(x)=x^q-x^{q^2}\beta$ and such polynomial has clearly less than $q^2$ roots. 
The same holds if $\beta+m \alpha=0$ as in this case $h_m(x)=mx+x^q$.
So, $L_{\alpha,\beta}$ is maximum scattered if and only if there is no $m \in \mathbb{F}_{q^5}^*$ with $\beta+m \alpha \ne 0$ such that the polynomial $k_m(x)=a_0 x + a_1 x^q-x^{q^2}$ has maximum kernel, where
$$a_0=\frac{m}{\beta+m \alpha}, \quad {\rm and} \quad a_1=\frac{1}{\beta+m \alpha}.$$

From \cite[Theorem 3.3 and Section 3.3]{CsMaPoZu2018} $k_m(x)$ has maximum kernel if and only if
\begin{equation} \label{maxker}
\begin{cases}
N\bigg( \frac{m}{\beta+m \alpha}\bigg)=-1, \\
\bigg( \frac{m}{\beta+m\alpha}\bigg)^q+\bigg( \frac{1}{\beta+m\alpha}\bigg)^{q+1}=\bigg(\frac{m}{\beta+m\alpha}\bigg)^{q^2+q+1}\bigg(\frac{1}{\beta+m\alpha}\bigg)^{q^3}.
\end{cases}
\end{equation}

Write $\lambda=m/(\beta+m\alpha)$, so that $m=\lambda \beta/(1-\lambda \alpha)$ and $1/(\beta+m\alpha)=(1-\lambda \alpha)/\beta$. We get that $L_{\alpha,\beta}$ is maximum scattered if and only if there is no $\lambda \in \mathbb{F}_{q^5}$ such that

\begin{equation}\label{rewr}
\begin{cases}
N(\lambda)=-1, \\
\lambda^q+\frac{(1-\lambda \alpha)^{q+1}}{\beta^{q+1}}=\lambda^{q^2+q+1} \frac{(1-\lambda \alpha)^{q^3}}{\beta^{q^3}}.
\end{cases}
\end{equation}
and this is equivalent to \eqref{char1}.
\end{proof}

Our aim is to show with the help of Lemma \ref{char} that if 
$\alpha^q/\beta^{q+1} \in \mathbb{F}_{q^5} \setminus \mathbb{F}_q$, $\beta\neq0$, 
then $L_{\alpha,\beta}$ is not maximum scattered.

Applying the same strategy as in the proof of Lemma \ref{lem1}, we write 
$\lambda=l\gamma+\sum_{i=i}^{4} l_i \gamma^{q^i}$ where $\{\gamma,\gamma^q,\ldots,\gamma^{q^4}\}$ is a normal basis of $\mathbb{F}_{q^5}$ over $\mathbb{F}_q$. 
In this way, the set of solutions of (\ref{char1}) coincides with the set of $\mathbb{F}_q$-rational points of an algebraic variety $\mathcal{V}$ in $\mathbb{A}^5(\mathbb{F}_q)$ given by
ten equations

\begin{equation} \label{normb}
C_i(l,l_1,l_2,l_3,l_4)=0,\qquad i=1,2,\ldots,10.
\end{equation}
Applying the same birational map as in the proof of Lemma \ref{lem1} the algebraic variety is $
\mathbb{F}_{q^5}$-isomorphic to
\begin{equation} 
\mathcal{V}_1:
\begin{cases}
l \cdot l_1 \cdot l_2 \cdot l_3 \cdot l_4=-1, \\ 
l_1 \beta^{q^3+q+1}+\beta^{q^3}(1-l_1 \alpha^q)(1-l \alpha)-l_2 \cdot l_1 \cdot l \cdot (1-l_3 \alpha^{q^3}) \beta^{q+1}=0.
\end{cases}
\end{equation}
Since in these variables the action of the Frobenius morphism is just a shift of coordinates, $\mathcal{V}_1$ is also 
isomorphic to 
\begin{equation} \label{altro}
\begin{cases}
EQ_1: \ l \cdot l_1 \cdot l_2 \cdot l_3 \cdot l_4=-1, \\ 
EQ_2: \ l_1 \beta^{q^3+q+1}+\beta^{q^3}(1-l_1 \alpha^q)(1-l \alpha)-l_2 \cdot l_1 \cdot l \cdot(1-l_3 \alpha^{q^3}) \beta^{q+1}=0,\\
EQ_{2q}: \ l_2 \beta^{q^4+q^2+q}+\beta^{q^4}(1-l_2 \alpha^{q^2})(1-l_1 \alpha^q)-l_3 \cdot l_2 \cdot l_1 \cdot (1-l_4 \alpha^{q^4}) \beta^{q^2+q}=0,\\
EQ_{2q^2}: \ l_3 \beta^{1+q^3+q^2}+\beta(1-l_3 \alpha^{q^3})(1-l_2 \alpha^{q^2})-l_4 \cdot l_3 \cdot l_2 \cdot (1-l \alpha) \beta^{q^3+q^2}=0,\\
EQ_{2q^3}: \ l_4 \beta^{q+q^4+q^3}+\beta^{q}(1-l_4 \alpha^{q^4})(1-l_3 \alpha^{q^3})-l \cdot l_4 \cdot l_3 \cdot (1-l_1 \alpha^{q}) \beta^{q^4+q^3}=0,\\
EQ_{2q^4}: \ l \beta^{q^2+1+q^4}+\beta^{q^2}(1-l \alpha)(1-l_4 \alpha^{q^4})-l_1 \cdot l \cdot l_4 \cdot (1-l_2 \alpha^{q^2}) \beta^{1+q^4}=0, \end{cases}
\end{equation}

Hence in the following we will prove that $L_{\alpha,\beta}$ with $\alpha^q/\beta^{q+1} \not\in \mathbb{F}_q$ is not maximum scattered proving that \eqref{normb} have an $\mathbb{F}_{q}$-rational
solution.
To this aim we will study the variety $\mathcal{V}_2$ proving that it is equivalent to an algebraic curve of degree 4. Since the dimension is a birational invariant this will show that also $\mathcal{V}$ is an algebraic curve. Showing that the curve of degree $4$ is absolutely irreducible of genus at most $3$, and using again that genus and irreducibility are invariant, we will obtain the same properties for $\mathcal{V}$. At this point, the existence of an $\mathbb{F}_{q^5}$-rational point of $\mathcal{V}$ will be ensured by the Hasse-Weil Theorem. 

According to this general strategy, we start with the following technical lemma.

\begin{sloppypar}
\begin{lemma} \label{prel1}
Let $\alpha,\beta \in \mathbb{F}_{q^5}$ with $\beta \ne 0$ and $\alpha^q/\beta^{q+1} \in \mathbb{F}_{q^5} \setminus \mathbb{F}_q$. Then the variety $\mathcal{V}_2$ (and hence also $\mathcal{V}$), is equivalent to the quartic curve
$\mathcal{Q}:~F(X,Y)=0$, where

\begin{gather}
 F(X,Y)=X^2  Y^2 \beta^{q^2+q+1} \alpha^{q^4+2q^3+q^2} - X^2  Y^2\beta \alpha^{q^4+2q^3+2q^2} -
            X^2  Y^2 \beta^{q^3+2q^2+q}\alpha^{q^3} \nonumber \\
+ X^2  Y^2 \beta^{q^3+q^2}\alpha^{q^3+q^2} +
            X^2  YN(\beta) \alpha^{q^3+q^2} - 2X^2  Y \beta^{q^2+q+1} \alpha^{q^4+q^3+q^2} \nonumber \\           
             -X^2  Y\beta^{q^4+q^3+1} \alpha^{q^3+2q^2}
+ 2X^2  Y \beta \alpha^{q^4+q^3+2q^2} +
            X^2  Y \beta^{q^3+2q^2+q} \nonumber \\
             - X^2  Y \beta^{q^3+q^2}\alpha^{q^2} - X^2N(\beta)\alpha^{q^2} +
            X^2\beta^{q^2+q+1}\alpha^{q^4+q^2} 
+ X^2\beta^{q^4+q^3+1}\alpha^{2q^2} \nonumber \\
 - X^2 \beta \alpha^{q^4+2q^2} +X  Y^2 \beta^{q^3+2q^2+q+1}\alpha^{q^4+q^3} - X  Y^2 \beta^{2q^3+q^2+q+1}\alpha^{q^4}\nonumber \\ 
          -  2 X  Y^2 \beta^{q^3+q^2+1}\alpha^{q^4+q^3+q^2} + 2 X  Y^2 \beta \alpha^{q^4+2q^3+q^2} + X  Y^2 \beta^{2q^3+2q^2}  \nonumber \\
          - X  Y^2 \beta^{q^3+q^2}\alpha^{q^3}
 + X  YN(\beta) \beta^{q^3+q^2} -   X  Y \beta^{q^3+2q^2+q+1}\alpha^{q^4} \nonumber \\
  - X  YN(\beta)\alpha^{q^3} + 2X  Y \beta^{q^2+q+1}\alpha^{q^4+q^3}
 - X  Y \beta^{q^4+2q^3+q^2+1} \alpha^{q^2}  \nonumber \\
 + 2 X Y \beta^{q^3+q^2+1}\alpha^{q^4+q^2} + 2 X  Y \beta^{q^4+q^3+1}\alpha^{q^3+q^2} - 4 X  Y \beta\alpha^{q^4+q^3+q^2}
\nonumber \\    
-    X  Y \beta^{2q^3+2q^2+q} \alpha - X  Y\beta^{q^4+2q^3+2q^2} \alpha^{q} +
            X  Y \beta^{q^3+q^2}N(\alpha) \nonumber \\
+ X  Y\beta^{q^3+q^2}
 + XN(\beta) 
-X \beta^{q^2+q+1} \alpha^{q^4} - 2 X \beta^{q^4+q^3+1} \alpha^{q^2}  \nonumber \\
+ 2 X \beta\alpha^{q^4+q2} +
            X \beta^{q^4+2q^3+q^2}\alpha^{q^2+q+1} - X \beta^{q^3+q^2} \alpha^{q^4+q^2+q+1} 
- Y^2\beta^{2q^3+2q^2+1}\alpha^{q^4} \nonumber \\
 +             2Y^2 \beta^{q^3+q^2+1} \alpha^{q^4+q^3} - Y^2 \beta\alpha^{q^4+2q^3} + Y \beta^{q^4+2q^3+q^2+1} - 2 Y \beta^{q^3+q^2+1} \alpha^{q^4}\nonumber \\
 - Y\beta^{q^4+q^3+1} \alpha^{q^3} + 2Y\beta \alpha^{q^4+q^3} +
       Y \beta^{2q^3+2q^2} \alpha^{q^4+q+1} - Y \beta^{q^3+q^2}\alpha^{q^4+q^3+q+1} \nonumber \\
             + \beta^{q^4+q^3+1} - \beta \alpha^{q^4} 
-\beta^{q^4+2q^3+q^2}\alpha^{q+1} +
 \beta^{q^3+q^2}\alpha^{q^4+q+1}.\label{quartica}
 \end{gather}
\end{lemma}
\end{sloppypar}

\begin{proof}
The following computations can be checked using MAGMA. The system of equations \eqref{altro}
admits a solution if and only if $l_4=-1/(l \cdot l_1 \cdot l_2\cdot l_3)$ and $EQ_{2}$,$EQ_{2q}$ and $EQ_{2q^i}$ evaluated at $(l,l_1,l_2,l_3,-1/(l \cdot l_1 \cdot l_2\cdot l_3))$ are equal to zero for all $i=2,\ldots,4$.
Clearly $l,l_1,l_2,l_3,l_4 \ne 0$.
Since 
$$EQ_{2q^2}(l,l_1,l_2,l_3,-1/(l \cdot l_1 \cdot l_2\cdot l_3)): \ l \cdot (l_1\cdot l_2 \cdot l_3 \cdot \beta \alpha^{q^3+q^2}-l_1\cdot l_2 \cdot \beta \alpha^{q^2}$$
$$+l_1 \cdot l_3 \cdot \beta^{q^3+q^2+1}-l_1 \cdot l_3 \cdot \beta \alpha^{q^3}+l_1 \beta-\beta^{q^3+q^2}\alpha)+\beta^{q^3+q^2}=0,$$
we get 
$$l_4=-\frac{1}{l \cdot l_1 \cdot l_2 \cdot l_3}, \ l=\frac{-\beta^{q^3+q^2}}{P(l_1,l_2,l_3)},$$
with
\[
\scriptsize
P(l_1,l_2,l_3)=l_1\cdot l_2 \cdot l_3 \cdot \beta \alpha^{q^3+q^2}-l_1\cdot l_2 \cdot \beta \alpha^{q^2}+l_1 \cdot l_3 \cdot \beta^{q^3+q^2+1}-l_1 \cdot l_3 \cdot \beta \alpha^{q^3}+l_1 \beta-\beta^{q^3+q^2}\alpha
\] 
and $EQ_{2}$,$EQ_{2q}$ and $EQ_{2q^i}$ evaluated at $(-\beta^{q^3+q^2}/P(l_1,l_2,l_3),l_1,l_2,l_3,-1/(l \cdot l_1 \cdot l_2\cdot l_3))$ are equal to zero for $i=3,4$.
Clearly $P(l_1,l_2,l_3) \ne 0$ as $\beta \ne 0$. 
Now, $EQ_{2q}=0$ implies
$$C_1(l_2,l_3)l_1+C_2(l_2,l_3)=0,$$
where

$$C_1(l_2,l_3)=l_2 l_3 \beta^{q+1}\alpha^{q^4+q^3+q^2} - l_2 \cdot l_3 \beta^{q^3+q^2+q} - l_2 \beta^{q+1}\alpha^{q^4+q^2}+l_2 \beta^{q^4+q^3}\alpha^{q^2+q} $$
$$+ l_3 \beta^{q^3+q^2+q+1}\alpha^{q^4} - l_3\beta^{q+1}\alpha^{q^4+q^2} + \beta^{q+1}\alpha^{q^4} 
        - \beta^{q^4+q^3}\alpha^q,$$
and
$$C_2(l_2,l_3)=  \beta^{q^3(}l_2 \beta^{q^4+q^2+q} - l_2 \beta^{q^4}\alpha^{q^2} - \beta^{q^2+q} \alpha^{q^4+1}+ \beta^{q^4}).$$
\begin{sloppypar}
We distinguish two cases: $C_1(l_2,l_3)=C_2(l_2,l_3)=0$ or $l_1=-C_2(l_2,l_3)/C_1(l_2,l_3)$. 
\end{sloppypar}
\begin{itemize}
\item \textbf{Case 1:  $C_1(l_2,l_3)=C_2(l_2,l_3)=0$.} Hence 
$$l_2=\frac{\beta^{q^2+q}\alpha^{q^4+1}-\beta^{q^4}}{\beta^{q^4+q^2+q}-\beta^{q^4}\alpha^{q^2}},$$
and
$$l_3=-P_2/P_1,$$
where
$$
P_1=N(\beta)\alpha^{q^4} - \beta^{q^4+q+1}\alpha^{q^4+q^3} + \beta^{q+1}\alpha^{2q^4+q^3+q^2+1} - \beta^{q^4+q^3+1}\alpha^{q^4+q^2} $$
$$- \beta^{q^4+q^2+q}\alpha^{q^4+1} + \beta^{q^4+q^3},
$$
$$
P_2=\beta^{q^4+q+1}\alpha^{q^4} -\beta^{q+1}\alpha^{2q^4+q^2+1} - \beta^{2q^4+q^3}\alpha^q + \beta^{q^4+q^3}\alpha^{q^4+q^2+q+1}.$$

Indeed if $P_1=P_2=0$ then $\alpha^q/\beta^{q+1} \in \mathbb{F}_q$. 
This fact can observed noting that from $P_2=0$, either $\beta=(\beta^{q^4+q^3} \alpha^q)/(\beta^q\alpha^{q^4})$ or $\beta=\alpha^{q^3+q+1}$. In the former case $\alpha^q/\beta^{q+1}=\alpha^{q^4}/\beta^{q^4+q^3}=(\alpha^q/\beta^{q+1})^{q^3}$. In the latter case $\alpha^q/\beta^{q+1}=1/N(\alpha)$. 
Substituting $l_2$ and $l_3$ in $EQ_{2q^3}$ we get

$$l_1Q_1+Q_2=0,$$
where

$$Q_1=\beta^{q^4+2q^3+q^2}(\beta^{q^4}-\alpha^{q^4+q^2+1})(\beta^{q+1}-\alpha^q)(\beta^{2q+q^2}\alpha^{q^4+1}  - \beta^{2q^4+q^3+q^2+q}\alpha^{q} $$
$$+ \beta^{2q^4+q}\alpha^{q^3+q} - \beta^{q^4+q}N(\alpha) -\beta^{q^4+q}+\beta^{2q^4+q^3}\alpha^{q^2+q}),$$
and
$$Q_2=\beta^{q^4+2q^3+q^2} (\beta^{q+1}-\alpha^q)^q (\beta^{1+2q+q^2+q^3+2q^4}\alpha^{q^4+1} -\beta^{1+2q+q^2+q^4}\alpha^{2q^4+1}$$
$$ -\beta^{1+2q+2q^4}\alpha^{1+q^3+q^4}+\beta^{1+2q+q^4}\alpha^{2+q^2+2q^4}
            -\beta^{1+q+q^3+2q^4}\alpha^{1+q^2+q^4} +\beta^{1+q+2q^4}\alpha^{q^4}$$
$$-
\beta^{2q+q^2+q^3+q^4}\alpha^{2+q^4} +\beta^{2q+q^2}\alpha^{2+2q^4} +\beta^{q+q^3+2q^4}\alpha +
       \beta^{q+2q^4}\alpha^{1+q+q^3+q^4}$$
$$-\beta^{q+q^4}\alpha^{q^4+1}N(\alpha)-\beta^{q+q^4}\alpha^{q^4} -\beta^{q^3+3q^4}\alpha^{q} +\beta^{q^3+2q^4}\alpha^{1+q+q^2+q^4}).$$

If $Q_1=Q_2=0$ then using again that $\alpha^q/\beta^{q+1} \not\in \mathbb{F}_q$, $\beta^{2q+q^2}\alpha^{q^4+1}  - \beta^{2q^4+q^3+q^2+q}\alpha^{q} + \beta^{2q^4+q}\alpha^{q^3+q} - \beta^{q^4+q}N(\alpha) -\beta^{q^4+q}+\beta^{2q^4+q^3}\alpha^{q^2+q}=0$ 
and
$\beta^{1+2q+q^2+q^3+2q^4}\alpha^{q^4+1} -\beta^{1+2q+q^2+q^4}\alpha^{2q^4+1}$
$ -\beta^{1+2q+2q^4}\alpha^{1+q^3+q^4}+\beta^{1+2q+q^4}\alpha^{2+q^2+2q^4}
            -\beta^{1+q+q^3+2q^4}\alpha^{1+q^2+q^4} +\beta^{1+q+2q^4}\alpha^{q^4}$
$-
\beta^{2q+q^2+q^3+q^4}\alpha^{2+q^4} +\beta^{2q+q^2}\alpha^{2+2q^4} +\beta^{q+q^3+2q^4}\alpha +
       \beta^{q+2q^4}\alpha^{1+q+q^3+q^4}$
$-\beta^{q+q^4}\alpha^{q^4+1}N(\alpha)-\beta^{q+q^4}\alpha^{q^4} -\beta^{q^3+3q^4}\alpha^{q} +\beta^{q^3+2q^4}\alpha^{1+q+q^2+q^4}=0$.

{If $\beta^{2q}\alpha^{q^4+1}- \beta^{2q^4+q^3+q}\alpha^q=0$ then also $ \beta^{2q^4+q}\alpha^{q^3+q} - \beta^{q^4+q}N(\alpha) -\beta^{q^4+q}+\beta^{2q^4+q^3}\alpha^{q^2+q}=0$ from the first equation. From the first equation $\beta^q=\beta^{2q^4+q^3} \alpha^q/\alpha^{q^4+1}$. Substituting to the second equation yields $0=\beta^{2q^4}\alpha^{q^3+q}-\beta^{q^4}N(\alpha)-\beta^{q^4}+\alpha^{q^4+q^2+1}=(\beta^{q^4}-\alpha^{q^4+q^2+1})(\beta^{q^4} \alpha^{q^3+q}-1)$. Hence $\alpha^q/\beta^{q+1} \in \mathbb{F}_q$, a contradiction.}

Hence $\beta^{q^2}=(-\beta^{q+2q^4}\alpha^{q^3+q} +\beta^{q^4+q} N(\alpha) +\beta^{q^4+q} - \beta^{2q^4+q^3}\alpha^{q^2+q})/(\beta^{2q}\alpha^{q^4+1} -\beta^{2q^4+q^3+q}\alpha^q)$ for the first equation. Substituting $\beta^{q^2}$ in the second equation we get that either $\beta^q \alpha - \beta^{q^4} \alpha^q=0$, or $\beta^q\alpha^{q^4+q^3+1}- \beta^{q^4+q^3}=0$, or $\beta^{q+1} \alpha^{q^4} - \beta^{q^4+q^3}\alpha^q=0$. 
If $\beta^q \alpha - \beta^{q^4} \alpha^q=0$ then substituting $\beta^{q^2}$ in the expressions of $l_3$ and $l_4$ above we get that $l$ and $l_4$ are function of $\alpha$ and $\beta$. This fact is compatible with the description already obtained for $l_4$ if and only if $-\beta^{2q^3+q}\alpha+ 2\beta^{q^3+q}\alpha^{q^4+q^3+q+1} -\beta^q \alpha^{2q^4+2q^3+2q+1}=\alpha\beta(\beta^{q^3}-\alpha^{q^4+q^3+q})$, which is not possible.
If the second case occurs then, subsituting again $l_2$, $l_3$ and $\beta^{q^4}$ we get that $\alpha^{q^4+1}\beta^q (\beta^q \alpha^{q^3}-\beta^{q^3} \alpha^{q^2})(\beta^q \alpha^{1+q+2q^3+q^4}\beta^{q^3})(b-\alpha^{q^3+q+1})=0$. In any case $\alpha^q/\beta^{q+1} \in \mathbb{F}_q$.
The last case implies that $\alpha^q/\beta^{q+1} \in \mathbb{F}_{q^3}$, so that again we get a contradiction.

This shows that $l_1=-Q_2/Q_1$. Substituting $l_1$ we get that all the conditions are satisfied. 
Hence the related point of $\mathcal{V}_2$ is
$$l_4=-\frac{1}{l \cdot l_1 \cdot l_2 \cdot l_3}, \ l=\frac{-\beta^{q^3+q^2}}{P(l_1,l_2,l_3)}, \ l_2=\frac{\beta^{q^2+q}\alpha^{q^4+1}-\beta^{q^4}}{\beta^{q^4+q^2+q}-\beta^{q^4}\alpha^{q^2}},$$
$$ l_3=\frac{-P_2}{P_1}, \ l_1=\frac{-Q_2}{Q_1}.$$
Substituting in $F(l_2,l_3)$ the value $l_2=\frac{\beta^{q^2+q}\alpha^{q^4+1}-\beta^{q^4}}{\beta^{q^4+q^2+q}-\beta^{q^4}\alpha^{q^2}}$ we get that $l_3=-P_2/P_1$ is a solution. This implies that $[l_1,l_2]$ is a point of the quartic.
\item \textbf{Case 2: $l_1=-C_2(l_2,l_3)/C_1(l_2,l_3)$.}
Substituting the expression of $l_1$ in $EQ_2$, $EQ_{2q^3}$ and $EQ_{2q^4}$ we get that all the conditions are satisfied once $F(l_2,l_3)=0$ (cf.\ (\ref{quartica})).

This shows that in any case a point of $\mathcal{V}_2$ corresponds uniquely to a point of the quartic $F(l_2,l_3)=0$.
Hence the variety $\mathcal{V}_2$ is a curve.

\end{itemize}
\end{proof}

Following the general strategy described before we are going to show that the quartic $\mathcal{Q}$ is absolutely irreducible. Since its genus is at most $g=(4-1)(4-2)/2=3$, and irreducibility, genus and dimension are birationally invariant, from Lemma \ref{prel1} and the Hasse-Weil bound we would obtain that the number of $\mathbb{F}_q$-rational points of $\mathcal{V}$ is at least:

$$q+1-2g\sqrt{q} \geq q+1-6\sqrt{q}>0$$
provided that $q \geq 37$. If $q<37$ it can be easily checked with MAGMA that the quartic $\mathcal{Q}$ has at least an $\mathbb{F}_{q^5}$-rational point of type $[\ell,\ell^q]$, implying by linearity of the other variables, a solution $[\ell,\ell^q,\ldots,\ell^{q^4}]$ of $\mathcal{V}_2$ and hence a solution of \eqref{char1}.

\begin{proposition} \label{nonMS}
Let $\alpha,\beta \in \mathbb{F}_{q^5}$ with $\beta \ne 0$ and $\alpha^q/\beta^{q+1} \in \mathbb{F}_{q^5} \setminus \mathbb{F}_q$. Then $L_{\alpha,\beta}$ is not maximum scattered.
\end{proposition}

\begin{proof}
From Lemma \ref{prel1} it is sufficient to show that the quartic $\mathcal{Q}$ is absolutely irreducible. 
The irreducibility of $\mathcal{Q}$ is equivalent to the non-existence of lines or quadrics as components.

\begin{itemize}
\item \textbf{Step 1: $\mathcal{Q}$ has not linear components.}
Suppose that $\mathcal{Q}$ is the product of a line and a (possibly reducible) cubic defined respectively by the affine polynomial:
$$L_1:=A_1 X+A_2 Y+A_3=0,$$
$$C_1:=B_1 X^3+B_2 Y^3+B_3 X^2 Y+B_4 Y^2 X+B_5 X^2+B_6 Y^2$$
$$+B_7 XY+B_8 X+B_9 Y+B_{10}=0.$$
Then forcing the polynomial $F-L_1\cdot C_1$ to be identically zero we get that $A_1 B_1=0$ and $A_1 B_3=-A_2 B_1$. If $A_1=0$ then since $A_2 \ne 0$ we get $B_1=0$ from the second equation. Thus, $B_1=0$ can be assumed. Analogously, from $A_2 B_2=0$ and $A_2 B_4=-A_1 B_2$ we get that $B_2=0$. Since $A_1 B_3=A_1 B_5=0$ we distinguish two cases.

\noindent\underline{Case 1.}
$A_1=0$. Since $A_2 \ne 0$ we get $B_6=0$. Since $A_2 B_9= \alpha^{q^4}(\beta^{q+1}-\alpha^{q})^{q^2} \beta$ we get that $A_2 B_9 \ne 0$ and $B_9$ can be written with respect to $A_2$. Substituting we get $\alpha^{q^2}\beta(\beta^{q+1}-\alpha^q)^{q^3+q^2}= A_3 B_5$, and hence $B_5$ can be written with respect to $A_3$ and $A_3 \ne 0$. From $A_2^2 A_3 B_4=0$ we get $B_4=0$. Other conditions that can be obtained at this point are

\scalebox{0.8}{\parbox{\linewidth}{
$$B_{10}=\frac{-\beta^{q^4+q^3+1} A_2 A_3 +\beta \alpha^{q^4}A_2 A_3 +\beta^{q^4+2q^3+q^2}\alpha^{q+1} A_2 A_3 -\beta^{q^4+q^2} \alpha^{q^4+q+1}A_2 A_3}{A_2 A_3^2},$$

$$B_3=\frac{-\beta^{q^2+q+1}\alpha^{q^4+2q^3+q^2}A_2 A_3 + \beta \alpha^{q^4+2q^3+2q^2}A_2 A_3 }{A_2^2 A_3}$$
$$\frac{+\beta^{q^3+2q^2+q}\alpha^{q^3} A_2 A_3 - \beta^{q^3+q^2}\alpha^{q^3+q^2}A_2 A_3}{A_2^2 A_3},$$

$$B_7=\frac{\beta^{q^3+2q^2+q+1}\alpha^{q^4+q^3}A_2 A_3 +\alpha^{q^4+2q^3}\beta^{q^2+q+1} A_2 A_3 +2\beta^{q^3+q^2+1} \alpha^{q^4+q^3+q^2}A_2 A_3 }{A_2^2 A_3}$$
$$\frac{- 2\beta \alpha^{q^4+2q^3+q^2} A_2 A_3 -\beta^{2q^3+2q^2} A_2 A_3 +\beta^{q^3+q^2}\alpha^{q^3} A_2 A_3}{A_2^2 A_3},$$

$$B_8=-\frac{N(\beta)A_2 -\beta^{q^2+q+1}\alpha^{q^4} A_2 - 2\beta^{q^4+q^3+1}\alpha^{q^2}A_2 + 2\beta \alpha^{q^4+q^2}A_2 }{A_2 A_3}$$
$$\frac{+\beta^{q^4+2q^3+q^2} \alpha^{q^2+q+1} A_2 -\beta^{q^3+q^2}\alpha^{q^4+q^2+q+1}A_2}{A_2 A_3}.$$}}

\noindent
Since we get $\beta^{q^3+2q^2+q+1} \alpha^{q^4+q^3} A_2 A_3=0$, we have a contradiction.


\scalebox{0.9}{\parbox{\linewidth}{\underline{Case 2.} $A_1 \ne 0$ and $B_3=B_5=0$. In this case,
$$B_8=\frac{N(\beta)\alpha^{q^2} -\beta^{q^2+q+1}\alpha^{q^4+q^2} -\beta^{q^4+q^3+1}\alpha^{2q^2}+ \beta \alpha^{q^4+2q^2}}{A_1},$$

$$B_4=\frac{-\beta^{q^2+q+1}\alpha^{q^+2q^3+q^2} + \beta \alpha^{q^4+2q^3+2q^2} +\beta^{q^3+2q^2+q}\alpha^{q^3} -\beta^{q^3+q^2}\alpha^{q^3+q^2}}{A_1},$$

$$B_7=-\frac{N(\beta)\alpha^{q^3+q^2}- 2\beta^{q^2+q+1}\alpha^{q^4+q^3+q^2} }{A_1}$$
$$\frac{-\beta^{q^3+q^2+1}\alpha^{q^3+2q^2} + 2\beta \alpha^{q^4+q^3+2q^2}+\beta^{q^3+2q^2+q} -\beta^{q^3+q^2}\alpha^{q^2}}{A_1}.$$

From the degree $3$ term $Y^3 \beta^{q^2+q+1}\alpha^{q^4+2q^3+q^2}A_2 - Y^3 \beta \alpha^{q^4+2q^3+2q^2} A_2 - Y^3 \beta^{q^3+2q^2+q} \alpha^{q^3} A_2 + Y^3 \beta^{q^3+q^2} \alpha^{q^3+q^2}A_2$ we get that $\beta \alpha^{q^4+q^3+q^2} - \beta^{q^3+q^2}=0$. Using this fact the expressions of $B_6$, $B_{10}$ and $B_9$ can be obtained with respect to $\alpha,\beta$ and the $A_i$'s and substituting them we get that either $A_2=0$ or $\beta^{q^3+q^2}\alpha^{q^4}A_1 - \beta^{q^4+q^3}\alpha^{q^2}A_2 +\alpha^{q^4+q^2}A_2 - \alpha^{^4+q^3}A_1=0$.

If $A_2=0$ then $A_3=-(\beta^{q^3+q^2}\alpha^{q^4+1}A_1-\alpha^{q^4+q^3+1}A_1)/(\beta^{q^4+q^3+q^2} -\alpha^{q^4+q^3+q^2+1})$ and subtituting we get $A_1^2 \alpha^{q^4+1} \beta^{2q^3+2q^2}(\beta^{q+1}-\alpha^q)^{q^3+q^2}(\beta^{q^4+q^3+q^2}- \alpha^{q^4+q^3+q^2})=0$, a contradiction.

Hence $A_2=-(\beta^{q^3+q^2}\alpha^{q^4}A_1- \alpha^{q^4+q^3}A_1)/(\beta^{q^4+q^3}\alpha^{q^2}+ \alpha^{q^4+q^2})$ and substituting  $A_1^2 \alpha^{q^4+1} \beta^{2q^3+2q^2}(\beta^{q+1}-\alpha^q)^{q^3+q^2}(\beta^{q^4+q^3+q^2}- \alpha^{q^4+q^3+q^2+1})=0$, a contradiction. 
This shows that $\mathcal{Q}$ has not a linear component.}}

\item \textbf{Step 2: $\mathcal{Q}$ is not the product of two irreducible quadrics.}
Assume that $\mathcal{Q}$ is the product of two absolutely irreducible quadrics defined by the affine polynomials
$$Q_1:=A_1 X^2 + A_2 Y^2 + A_3 X Y + A_4 X + A_5 Y - A_6=0,$$
$$Q_2:=B_1 X^2 + B_2 Y^2 + B_3 X Y + B_4 X + B_5 Y - B_6=0.$$

We force the bivariate polynomial $F(X,Y)-Q_1 \cdot Q_2$ to be identically zero.

From the coefficients of the polynomial $P$ we see that without loss of generality $A_2=0$ and $B_1=0$. Indeed $A_1 B_1=0$ and if $A_1=0$ using that $A_3 \ne 0$ we get $B_1=0$ from $-A_3 B_1=0$. From $A_5B_2=A_3B_2=A_1B_3=A_1B_4=0$ we get that either $B_2=0$ and since $B_3 \ne 0$, also $A_1=0$, or $B_2 \ne 0$, $A_3=A_5=0$ and from $A_1 \ne 0$ also $B_3=B_4=0$. We note that the latter case canno occur since otherwise $Q_1$ and $Q_2$ would be univariate polynomials and hence not curves. From $\beta^{2q^3+2q^2+1}\alpha^{q^4} - 2\beta^{q^3+q^2+1}\alpha^{q^4+q^3}+ \beta \alpha^{q^4+2q^3} - A_5 B_5=0$ and $\beta^{2q^3+2q^2+1}\alpha^{q^4} - 2\beta^{q^3+q^2+1}\alpha^{q^4+q^3}+ \beta \alpha^{q^4+2q^3}=\alpha^{q^4}\beta(\beta^{q^3+q^2}-\alpha^{q^3}) \ne 0$, we get that $A_5 \ne 0$ and $B_5 \ne 0$. In particular, $B_5$ can be written as a function of $A_5$, $\alpha$ and $\beta$. Substituting in $P(X,Y)$ we get that $\beta^{q^4+q^3+q^2+q+1}\alpha^{q^2} - \beta^{q^2+q+1}\alpha^{q^4+q^2} - \beta^{q^4+q^3+1}\alpha^{2q^2}+ \beta \alpha^{q^4+2q^2}- A_4 B_4=0=\alpha^{q^2}\beta (\beta^{q^4+q^3}-\alpha^{q^4})(\beta^{q^2+q}-\alpha^{q^2})-A_4 B_4$ and $\alpha^{q^2}\beta (\beta^{q^4+q^3}-\alpha^{q^4})(\beta^{q^2+q}-\alpha^{q^2}) \ne 0$ we getas before that $A_4 \ne 0$, $B_4 \ne 0$ and $B_4$ can be written as a function of $A_4,\alpha$ and $\beta$. Substituting in $P(X,Y)$ we get in the same way that $A_3 \ne 0$ and $B_3$ is a function of $A_3,\alpha$ and $\beta$ and $B_6$ is a function of $A_5^2$, $\alpha$ and $\beta$. From $( \beta^{q^3+q^2}\alpha^{q^4}A_6 + \beta^{q^4+q^3}A_5 - \alpha^{q^4+q^3}A_6 - \alpha^{q^4}A_5)(\beta^{q^3+q^2+1}A_6 - \beta \alpha^{q^3}A_6 - \beta A_5 + \beta^{q^3+q^2}\alpha^{q+1}A_5)=0$ we distinguish two subcases.

 \noindent\underline{Case 1.}
$ \beta^{q^3+q^2}\alpha^{q^4}A_6 + \beta^{q^4+q^3}A_5 - \alpha^{q^4+q^3}A_6 - \alpha^{q^4}A_5=0$.
Here we can write $A_6$ as a function of $A_5, \alpha$ and $\beta$ and substituting in $P(X,Y)$ we have that either $A_3=-\alpha^{q^3}A_4$, or $\beta^{q^4+q^3+1}\alpha^{q^2}A_3 - \beta \alpha^{q^4+q^3+q^2}A_4 - \beta\alpha^{q^4+q^2}A_3 + \beta^{q^3+q^2}A_4=0$. 
In the former case substituting in $P(X,Y)$ we get the following two necessary conditions:

$$P_1:=(\beta^{q^2+q}A_5 - \beta^{q^3+q^2}A_4 - \alpha^{q^2}A_5 + \alpha^{q^2}A_4)$$

$$(\beta^{q^3+q^2+1}\alpha^{q^4+q^3}A_4 + \beta \alpha^{q^4+q^3+q^2}A_5 - \beta \alpha^{q^4+2q^3}A_4 - \beta^{q^3+q^2}A_5)=P_{1,1}P_{1,2}=0;$$

$$P_2:=(\beta^{q^3+q^2}\alpha^{q^4}A_4 - \beta^{q^4+q^3}\alpha^{q^2}A_5 + \alpha^{q^4+q^2}A_5 - \alpha^{q^4+q^3}A_4)$$

$$(\beta^{q^4+q^3+q^2+q+1}A_5 - \beta{q^2+q+1}\alpha^{q^4}A_5 + \beta^{q^3+q^2+1}\alpha^{q^4}A_4 - \beta^{q^4+q^3+1}\alpha^{q^2}A_5 $$
$$+
            \beta \alpha^{q^4+q^2}A_5 - \beta \alpha^{q^4+q^3}A_4 - \beta^{2q^3+2q^2}\alpha^{q^4+q+1}A_4 +
            \beta^{q^3+q^2}\alpha^{q^4+q^3+q+1}A_4)$$
$$=P_{2,1} P_{2,2}=0.$$

From the factorization of the resultant of $P_1$ and $P_2$ with respect to $A_4$ we get that

$$C_1=N(\beta)\beta \alpha^{q^3}- \beta^{q^2+q+2}\alpha^{q^4+q^3}$$
$$- \beta^{q^4+q^3+2}\alpha^{q^3+q^2}+
        \beta^{q^3+q^2+1}N(\alpha)$$
$$+ \beta^{q^3+q^2+1} -\beta^{2q^3+2q^2}\alpha^{q+1}=0.$$

This case can be excluded as follows. 

Since both the resultant of $P_{1,1}$ with $P_{2,1}$ and $P_{2,2}$ with respect to $A_4$ cannot vanish we get that $P_{1,2}=\beta^{q^3+q^2+1}\alpha^{q^4+q^3}A_4 + \beta \alpha^{q^4+q^3+q^2}A_5 - \beta \alpha^{q^4+2q^3}A_4 - \beta^{q^3+q^2}A_5=0$ so that $A_5$ can be written as a function of $A_4,\alpha$ and $\beta$. Substituting $A_4$ in $P(X,Y)$ gives

$$C_2:=\beta^{q^4+q+3}\alpha^{q^4+3q^3+q^2}- \beta^{q^3+q^2+1}N(\beta)\alpha^{q^3} + \beta^{q^3+2q^2+q+2}\alpha^{q^4+q^3}$$
$$ -\beta^{q^3+q^2+q+2}\alpha^{q^4+2q^3+q^2+1}
 -\beta^{q^2+q+2}\alpha^{q^4+2q^3}+\beta^{q^4+2q^3+q^2+2} \alpha^{q^3+q^2}$$
$$
            - \beta^{q^4+q^3+q^2+2}\alpha^{q^4+2q^3+q^2+q} -\beta^{q^4+q^3+2}\alpha^{2q^3+q^2}+
            \beta^{2q^3+2q^2+q+1}\alpha^{q^3+1}$$
$$+ \beta^{q^4+2q^3+2q^2+1}\alpha^{q^3+q}- \beta^{2q^3+2q^2+1}+
            \beta^{q^3+q^2+1}N(\alpha) \alpha^{q^3} + \beta^{q^3+q^2+q}\alpha^{q^3}$$
$$ -\beta^{2q^3+2q^2}\alpha^{q^3+q+1}=0.$$

Comparing the resultants of $C_1$ and $C_2$ with respect to $\beta^q$, $\beta^{q^4}$ and $\alpha$ gives 

$$C_3:=N(\beta)\alpha^{q^3} - \beta^{q^2+q+1}\alpha^{q^4+q^3} - \beta^{q^4+q^3+1}\alpha^{q^3+q^2}+ \beta^{q^4+1}\alpha^{q^4+2q^3+q^2+q}$$
$$-
        \beta^{q^4+q^3+q^2}\alpha^{q^3+q}+ \beta^{q^3+q^2}=0,$$

which is not possible as the resultant of $C_3$ and $C_1$ with respect t $\beta^q$ cannot vanish.
Hence the second case occurs and $A_3$ can be written with respect to $A_4,\alpha$ and $\beta$. Substituting in $P(X,Y)$ the resulting expression of $A_4$ with respect to $A_5,\alpha$ and $\beta$ we get again that both $C_1=0$ and $C_2=0$ hold. Hence a contradiction can be obtained as in the previous case.

 \noindent\underline{Case 2.}
 $\beta^{q^3+q^2+1}A_6 - \beta \alpha^{q^3}A_6 - \beta A_5 + \beta^{q^3+q^2}\alpha^{q+1}A_5=0$. Substituting the expression of $A_6$ with respect to $A_5$, $\alpha$ and $\beta$ we get that either $\beta^{q^2+q}\alpha^{q^3}A_5 + \beta^{q^3+q^2}A_3 - \alpha^{q^3+q^2}A_5 - \alpha^{q^3}A_3=0$, or
$\beta^{q^3+q^2+1}\alpha^{q^4}A_3 - \beta \alpha^{q^4+q^3+q^2}A_5 - \beta \alpha^{q^4+q^3}A_3 + \beta^{q^3+q^2}A_5=0$. In the former case we can write $A_5$ with respect to $\alpha,\beta$ and $A_3$ getting again that $C_1=0$ and $C_2=0$ hold, a contradiction. In the second case again $A_5$ can be written as a function of $A_3$, $\alpha$ and $\beta$ and $A_3=\alpha^{q^3} A_4$. Since the  necessary conditions  
$N(\alpha)=1$, 
$$N(\beta)\beta^{q^3}\alpha^{q^4+q^3} - \beta^{q^3+q^2+q+1}\alpha^{2q^4+q^3}
 - \beta^{q^4+q^3+q+1}\alpha^{q^4+2q^3}$$
$$+\beta^{q^3+q+1}\alpha^{2q^4+2q^3+q^2+1}- \beta^{q+1}\alpha^{3q^4+3q^3+q^2+q+1}+ \beta^{q+1}\alpha^{2q^4+2q^3}$$
$$ - \beta^{q^4+2q^3+1}\alpha^{q^4+q^3+q^2}+ \beta^{q^4+q^3+1}\alpha^{2q^4+2q^3+q^2+q}- \beta^{2q^3+q^2+q}\alpha^{q^4+q^3+1}$$
$$+\beta^{q^3+q^2+q} \alpha^{2q^4+2q^3+q+1} - \beta^{q^4+2q^3+q^2}\alpha^{q^4+q^3+q}+ \beta^{2q^3+q^2}\alpha^{q^4}+ \beta^{q^4+2q^3}\alpha^{q^3}$$
$$ -\beta^{q^3}\alpha^{q^4+q^3}=0,$$

$$N(\beta)\beta^{q^3}\alpha^{q^4+q^3}- \beta^{q^3+q^2+q+1}\alpha^{2q^4+q^3} - \beta^{q^4+q^3+q+1}\alpha^{q^4+2q^3}+\beta^{q^3+q+1}\alpha^{2q^4+2q^3+q^2+1}$$
$$- \beta^{q+1}\N(\alpha)\alpha^{q^4+q^3}+ \beta^{q+1}\alpha^{2q^4+2q^3} -
       \beta^{q^4+2q^3+1}\alpha^{q^4+q^3+q^2} + \beta^{q^4+q^3+1}N(\alpha)$$
$$- \beta^{2q^3+q^2+q}\alpha^{q^4+q^3+1}+\beta^{q^3+q^2+q}N(\alpha)\alpha^{q^4+q^3}
     - \beta^{q^4+2q^3+q^2}\alpha^{q^4+q^3+q}$$
$$+ \beta^{2q^3+q^2}\alpha^{q^4}+\beta^{q^4+2q^3}\alpha^{q^3}-\beta^{q^3}\alpha^{q^4+q^3}=0$$
 
cannot hold simultaneously we get a contradiction.

\end{itemize} 

\end{proof}

\section{On the equivalence of $L_{\alpha,\beta}$ with known linear sets in $\PG(1,q^5)$}

According to \cite[Proposition 5.1 and Theorem 5.4]{CsMaPo2018a} the maximum scattered $\mathbb{F}_q$-linear set $L_{\alpha,\beta}$ is equivalent to some 
$L_{s}^\eta$ (see Sect.\ 1 for its definition) 
if and only if there exist $A,B,C,D,\lambda \in \mathbb{F}_{q^5}$ with $AD-BC \ne 0$, $\lambda \ne 0$ and $\tau$ automorphism of $\mathbb{F}_{q^5}$ such that 

\begin{equation} \label{eq}
\begin{pmatrix} A & B \\ C & D \end{pmatrix} \begin{pmatrix} x^{\tau} -\alpha^\tau x^{q^2 \tau} \\ x^{q \tau} -\beta^\tau x^{q^2 \tau} \end{pmatrix}=\begin{pmatrix} z \\ f_{s,\eta}(\lambda z)\end{pmatrix}
\end{equation}
where $f_{s,\eta}(z)=\eta z^{q^s}+z^{q^{5-s}}$.
We note that it is sufficient to consider the case $\lambda=1$ as
$$\begin{pmatrix} z \\ f_{s,\eta}(\lambda z)\end{pmatrix}=\begin{pmatrix} \lambda^{-1} & 0 \\ 0 & 1 \end{pmatrix}\begin{pmatrix} x \\ f_{s,\eta}(x)\end{pmatrix},$$
with $z=\lambda^{-1} x$.
We first deal with the case $s=1$. Then defining $y=x^{\tau}$, \eqref{eq} reads 

$$\begin{pmatrix} A & B \\ C & D \end{pmatrix} \begin{pmatrix} y -\alpha^\tau y^{q^2} \\ y^{q} -\beta^\tau y^{q^2} \end{pmatrix}=\begin{pmatrix} z \\ \eta z^q+z^{q^4}\end{pmatrix}.$$ 

Hence,
\begin{equation} \label{eqLP}
\begin{cases} A=0, \\ \beta^{\tau q} B^q \eta=0, \\ C \alpha^\tau+D \beta^\tau=-\eta B^q, \\ D= -\beta^{q^4 \tau} B^{q^4}, \\ C=B^{q^4}.\end{cases}
\end{equation}
As stated in section \ref{cforms}, 
if $\beta=0$ then $L_{\alpha,\beta}$ is equivalent to $L_{1}^\eta$. 
If $\beta \ne 0$ then from   \eqref{eqLP} $B=0=A$, which is not possible. The following Lemma is now proved.

\begin{lemma} \label{equivLP}
Let $L_{\alpha,\beta}$ be maximum scattered. Then $L_{\alpha,\beta}$ is equivalent to $L_{1}^\eta$ for some $\eta \in \mathbb{F}_{q^5}$ wih $N(\eta) \ne 1$ if and only if $\beta=0$ and $N(\alpha) \ne -1$. 
\end{lemma}

We now analyze the case $s=2$. If
$$\begin{pmatrix} A & B \\ C & D \end{pmatrix} \begin{pmatrix} x^{\tau} -\alpha^\tau x^{q^2 \tau} \\ x^{q \tau} -\beta^\tau x^{q^2 \tau} \end{pmatrix}=\begin{pmatrix} z \\ \eta z^{q^2}+z^{q^3}\end{pmatrix},$$
then

\begin{equation} \label{eqS}
\begin{cases} 
C=-A^{q^3} \alpha^{q^3 \tau}-B^{q^3} \beta^{q^3 \tau}, \\ D=0, \\ \eta A^{q^2} + C \alpha^{\tau}=0, \\ A^{q^3}+\eta B^{q^2} =0, \\ B^{q^3}-\eta [A^{q^2} \alpha^{q^2 \tau}+B^{q^2} \beta^{q^2 \tau}]=0.
\end{cases} 
\end{equation}

Define $A^{q^3}=-\eta B^{q^2}$ so that $A^{q^2}=-\eta^{q^4} B^q$. 
Since $\alpha= 0$ would imply the contradiction $A=B=0$,
we can also define $C=-(\eta A^{q^2})/ \alpha^\tau=(\eta^{q^4+1}B^q)/\alpha^\tau$. At this point   \eqref{eqS} reads

$$
\begin{cases}
\eta^{q^4+1} B^q-\eta B^{q^2} \alpha^{(q^3+1) \tau}+\alpha^\tau \beta^{q^3 \tau} B^{q^3}=0, \\ B^{q^3}+\eta^{q^4+1} \alpha^{q^2 \tau} B^q -\eta B^{q^2} \beta^{q^2 \tau}=0,
\end{cases}
$$
which is equivalent to require that the polynomials
\begin{equation} \label{eqS2}
\begin{cases}
 P_1(B):=B^{q^2}-\eta^{q^4} \beta^{q \tau} B^q+\eta^{q^4+q^3} \alpha^\tau B=0, \\
P_2(B):=\alpha^{q^4 \tau} \beta^{q^2 \tau} B^{q^2} - \eta^{q^4} \alpha^{(q^4+q^2) \tau} B^q + \eta^{q^4+q^3} B=0,
\end{cases}
\end{equation}
have at least one common root $B \in \mathbb{F}_{q^5}^*$. 
Since $\alpha,\beta\neq0$, \eqref{eqS2} is equivalent to

\begin{equation} \label{eqS21}
\begin{cases}
\beta^{q^2 \tau} \alpha^{q^4 \tau} P_1(B)-P_2(B)=0, \\ P_2(B)=0.
\end{cases}
\end{equation}

Since $\beta^{q+1} \ne \alpha^{q}$ as otherwise $L_{\alpha,\beta}$ is of pseudoregulus type, \eqref{eqS21} reads
$$\begin{cases} \alpha^{q^4 \tau} \beta^{q^2 \tau} B^{q^2} - \eta^{q^4} \alpha^{(q^4+q^2) \tau} B^q + \eta^{q^4+q^3} B=0, \\ -B^q+
\frac{\eta^{q^3}(\beta^{q^2 \tau}   \alpha^{(q^4+q) \tau}-1)}{\alpha^{q^4 \tau}(\beta^{(q+1) \tau}-\alpha^{q \tau})^q} 
B=0.\\ \end{cases}$$

Since in general $-B^q+kB=0$, $B\neq0$ implies $N(k)=1$, we obtain

\begin{equation} \label{eqS22}
\begin{cases}
\alpha^{q^4 \tau} \beta^{q^2 \tau} B^{q^2} - \eta^{q^4} \alpha^{(q^4+q^2) \tau} B^q + \eta^{q^4+q^3} B=0, \\ -B^q+\frac{\eta^{q^3}(\beta^{q^2 \tau}   \alpha^{(q^4+q) \tau}-1)}{\alpha^{q^4 \tau}(\beta^{(q+1) \tau}-\alpha^{q \tau})^q} B=0, \\ 
N \bigg( \frac{\eta^{q^2}(\beta^{q \tau}   \alpha^{(q^3+1) \tau}-1)}{\alpha^{q^3 \tau}(\beta^{(q+1) \tau}-\alpha^{q \tau})} 
\bigg)=1.
\end{cases}
\end{equation}

Hence we write
$$\eta^{q^2}=\lambda \bigg(
\frac{ \alpha^{q^3}(\beta^{(q+1)}-\alpha^{q})}{\beta^{q} \alpha^{q^3+1}-1}
\bigg)^\tau$$
where $N(\lambda)=1$, so that $B^q=\lambda^q B$. Substituting in $P_2$ and recalling that $B \ne 0$ we get

$$\alpha^{q^4 \tau} \beta^{q^2 \tau} \lambda^{q^2+q}-\eta^{q^4} \alpha^{(q^4+q^2)\tau} \lambda^q+\eta^{q^4+q^3}=0$$
and taking the $q^4$-power and dividing by $\lambda^{q+1}$ 

$$\alpha^{q^3 \tau} \beta^{q \tau}-\bigg( 
\frac{ \alpha^{q^3}(\beta^{(q+1)}-\alpha^{q})}{\beta^{q} \alpha^{q^3+1}-1}
\bigg)^{q \tau} \alpha^{(q^3+q) \tau} + \bigg( 
\frac{ \alpha^{q^3}(\beta^{(q+1)}-\alpha^{q})}{\beta^{q} \alpha^{q^3+1}-1}
\bigg)^{(q+1)\tau}=0,$$
which is equivalent to
\begin{equation} \label{condSh}
\alpha^{q^4} \beta^{q^2+q+1}-\alpha^{q^4+q^2} \beta - \alpha^{q^4+q} \beta^{q^2} + N(\alpha)-\beta^q \alpha^{q^3+1}+1=0.
\end{equation}

Furthermore if the previous condition is satisfied, recalling our definiton of $\eta$, then $L_{\alpha,\beta}$ is maximum scattered if and only if 
$$N\bigg( \frac{ \alpha^{q^3}(\beta^{(q+1)}-\alpha^{q})}{\beta^{q} \alpha^{q^3+1}-1}
\bigg) \ne 1.$$

This proves the following lemma.

\begin{lemma} \label{equivSh}
A linear set $L_{\alpha,\beta}$ is equivalent to $L_{2}^\eta$ for some $\eta\in\Fqq$ if and only if
\textup{(\ref{condSh})} holds.

If this is the case then $L_{\alpha,\beta}$ is maximum scattered if and only if 
$$N\bigg( \frac{ \alpha^{q^3}(\beta^{(q+1)}-\alpha^{q})}{\beta^{q} \alpha^{q^3+1}-1}
\bigg) \ne 1.$$
\end{lemma}

\begin{remark}
It can be checked with MAGMA or GAP that using Lemmas \ref{equivLP} and \ref{equivSh} then no new maximum scattered linear sets of type $L_{\alpha,\beta}$ can be obtained for $q\le11$.
\end{remark}

From $N(\alpha)+1 \in \mathbb{F}_q$, we note that a necessary condition for   \eqref{condSh} to hold is that 
$$\alpha^{q^4} \beta^{q^2+q+1}-\alpha^{q^4+q^2} \beta - \alpha^{q^4+q} \beta^{q^2} -\beta^q \alpha^{q^3+1}=\alpha \beta^{q^3+q^2+q}-\alpha^{q^3+1} \beta^q - \alpha^{q^2+1} \beta^{q^3} -\beta^{q^2} \alpha^{q^4+q},$$
which is equivalent to
$$\alpha^{q^4} \beta (\beta^{q^2+q}-\alpha^{q^2})=\alpha \beta^{q^3}(\beta^{q^2+q}-\alpha^{q^2}).$$
Since $\beta^{q+1} \ne \alpha^q$ from Lemma \ref{equivPSE} we get $\alpha^{q^4} \beta-\alpha \beta^{q^3}=0$ and hence
\begin{equation} \label{sh1}
\alpha^q/\beta^{q+1} \in \mathbb{F}_{q}^*.
\end{equation}

Hence let $\alpha^q/\beta^{q+1}=\lambda \in \mathbb{F}_{q}^*$. In this case   \eqref{condSh} reads
\begin{equation} \label{sh2}
\lambda^5N(\beta)^2+\lambda(1-3\lambda)N(\beta)+1=0.
\end{equation}

Thus, from Lemma \ref{equivSh} if $\alpha$ and $\beta$ satisfy   \eqref{sh1} and \eqref{sh2} $L_{\alpha,\beta}$ is equivalent to $L_{2}^\eta$ with $\eta=\alpha^{q^3}(\beta^{(q+1)}-\alpha^{q})/(\beta^q \alpha^{q^3+1}-1)$.
It follows that $L_{\alpha,\beta}$ is maximum scattered if and only if 
$$N(\eta) =N\bigg( \frac{\alpha^{q^3}(\beta^{(q+1)}-\alpha^{q})}{\beta^q \alpha^{q^3+1}-1}\bigg)= N\bigg(\frac{\lambda \beta^{q^3+q^2}(\beta^{q+1}-\lambda \beta^{q+1})}{\beta^q \lambda^2 \beta^{q^4+q^3+q^2+1}-1}\bigg)=$$
$$N(\beta)^4 \cdot N\bigg( \frac{\lambda(1-\lambda)}{\lambda^2N(\beta)-1}\bigg)\ne 1.$$
Since $\lambda \in \mathbb{F}_q^*$ we get that equivalently
\begin{equation} \label{ms}
\lambda^5(1-\lambda)^5 N(\beta)^4 \ne (\lambda^2 N(\beta)-1)^5.
\end{equation}
Computing the resultant of the polynomials $\lambda^5(1-\lambda)^5 Y^4-(\lambda^2 Y-1)^5$ and $\lambda^5 Y^2+\lambda(1-3\lambda)Y+1$ with respect to $Y$ we get $\lambda^{14}(\lambda-1)^{10}$. Hence if $\lambda^5(1-\lambda)^5 N_{q^5/q}(\beta)^4= (\lambda^2 N_{q^5/q}(\beta)-1)^5$ then either $\lambda=0$ or $\lambda=1$. If $\lambda=0$ then $\alpha=0$,
contradicting \eqref{condSh}. 
If $\lambda=1$ then $L_{\alpha,\beta}$ is of pseudoregulus type, a contradiction. 

Thus the following remark holds.

\begin{remark} \label{remSh}
A  linear set of type $L_{\alpha,\beta}$ is equivalent to $L_{2}^\eta$ if and only if $\alpha^q/\beta^{q+1}=\lambda \in \mathbb{F}_{q}^* \setminus \{1\}$ and \eqref{sh2} holds. If it is the case then $L_{\alpha,\beta}$ is maximum scattered.

\end{remark}

Summarizing, the following theorem collects all the possibile equivalences of maximum scattered linear sets of type $L_{\alpha,\beta}$ and known linear sets.

\begin{theorem} \label{equivAll}
Let $L_{\alpha,\beta}$ be scattered, $\alpha\beta\neq0$, $\alpha^q\neq\beta^{q+1}$.
Then
\begin{itemize}
\item  $\lambda=\alpha^q/\beta^{q+1}\in\Fq$;
\item $L_{\alpha,\beta}$ is equivalent (up to collineations) 
neither to $L_{1}^\eta$ for any $\eta$, nor to a linear set of pseudoregulus type;
\item $L_{\alpha,\beta}$ is equivalent to the Sheekey linear set $L_{2}^\eta$ for some $\eta$
if and only if $\lambda^5\N(\beta)^2+\lambda(1-3\lambda)\N(\beta)+1=0$; so,
\item
if $\lambda^5\N(\beta)^2+\lambda(1-3\lambda)\N(\beta)+1\neq0$, 
then $L_{\alpha,\beta}$ is of a new type; this does not occur for $q\le11$.
\end{itemize}
\end{theorem}  

\begin{remark}
Even though every maximum scattered linear set either of pseudoregulus type or of Lunardon-Polverino type is $L_{\alpha,\beta}$ for some $\alpha,\beta \in \mathbb{F}_{q^5}$, the same statement is not true in general for Sheekey linear sets $L_{2}^\eta$ with $N(\eta) \ne 1$. 

Indeed let $\eta \in \mathbb{F}_{q^5}^*$ such that $N(\eta)^2-N(\eta)+1=0$. 
This implies that $q \not\equiv 2 \pmod 3$. From Theorem \ref{equivAll} we want to show that there are no $\lambda \in \mathbb{F}_q^* \setminus \{1\}$ and $\beta \in \mathbb{F}_{q^5}^*$ such that

$$\begin{cases}\eta=\frac{\lambda \beta^{q^3+q^2}(\beta^{q+1}-\lambda \beta^{q+1})}{\beta^q \lambda^2 \beta^{q^4+q^3+q^2+1}-1}, \\  \lambda^5N(\beta)^2+\lambda(1-3\lambda)N(\beta)+1=0.\end{cases}$$

Suppose by contradiction that $\eta=\frac{\lambda \beta^{q^3+q^2}(\beta^{q+1}-\lambda \beta^{q+1})}{\beta^q \lambda^2 \beta^{q^4+q^3+q^2+1}-1}$ and $\lambda^5N(\beta)^2+\lambda(1-3\lambda)N(\beta)+1=0$. Then as in \eqref{ms}, we have that 
$$N(\eta)=N(\beta)^4 \cdot N_{q^5/q}\bigg( \frac{\lambda(1-\lambda)}{\lambda^2N(\beta)-1}\bigg)=N(\beta)^4 \frac{\lambda^5 (1-\lambda)^5}{((\lambda^2N(\beta)-1)^5}$$
and hence $N(\eta)^2-N(\eta)+1=0$ implies
$$N(\beta)^8 \lambda^{10} (1-\lambda)^{10}-N(\beta)^4 \lambda^5(1-\lambda)^5(\lambda^2 N(\beta)-1)^5+(\lambda^2 N(\beta)-1)^{10}=0.$$
Since the resultant of the polynomials $P_1(\lambda,N)=N^8 \lambda^{10} (1-\lambda)^{10}-N^4 \lambda^5(1-\lambda)^5(\lambda^2 N-1)^5+(\lambda^2 N-1)^{10}$ and $P_2(\lambda,N)=\lambda^5N^2+\lambda(1-3\lambda)N+1$ with respect to $N$ is  $\lambda^{28}(\lambda-1)^{22}$ and $\lambda \in \mathbb{F}_q^* \setminus \{1\}$ we have a contradiction. From Proposition \ref{height Sh} the cases $N(\eta)^2-N(\eta)+1=0$ are exactly those for which $\Lambda \cap \Lambda^\sigma$ has not height four. This explicit construction is hence consistent with Proposition \ref{height4}.
\end{remark}

We end this section with the following question.
\begin{question}
It has been proven in Proposition \ref{nonMS} that 
for $\beta\ne0$ $\alpha^q/\beta^{q+1} \in \mathbb{F}_{q}$ is a necessary condition for $L_{\alpha,\beta}$ to be maximum scattered. From Lemma \ref{equivLP}, $L_{\alpha,\beta}$ is equivalent to $L_1^{\eta}$ for some $\eta$ with $N(\eta) \ne 1$ if and only if $\beta=0$, while if $\alpha^q/\beta^{q+1}=1$ then $L_{\alpha,\beta}$ is of pseudoregulus type. From Remark \ref{remSh} $L_{\alpha,\beta}$ is equivalent to $L_{2,\eta}$ for some $\eta$ with $N(\eta) \ne 1$ if and only if  $\lambda=\alpha^q/\beta^{q+1} \in \mathbb{F}_{q}^* \setminus \{1\}$ and \eqref{sh2} holds. Is it true that $L_{\alpha,\beta}$ with $\alpha^q/\beta^{q+1} \in \mathbb{F}_{q}^* \setminus \{1\}$ is maximum scattered if and only if it is equivalent to $L_2^{\eta}$ for some $N(\eta) \ne 1$? If the answer to this question is negative then the family of $L_{\alpha,\beta}$ contains new maximum scattered linear sets. Otherwise, it would provide a new characterization of the known maximum scattered linear sets in $\PG(1,q^5)$.
\end{question}

\bigskip

\noindent
Maria Montanucci\\
Technical University of Denmark\\
Asmussens All\'e\\
Building 303B, room 150\\
2800 Kgs. Lyngby\\
Denmark

\bigskip

\noindent
Corrado Zanella\\
Dipartimento di Tecnica e Gestione dei Sistemi Industriali\\
Universit\`a degli Studi di Padova\\
Stradella S. Nicola, 3\\
36100 Vicenza VI\\
Italy


\begin{thebibliography}{pippo}

\bibitem{BaGiMaPo2018}
{\sc D. Bartoli - M. Giulietti - G. Marino - O. Polverino:} 
Maximum scattered linear sets and complete caps
in Galois spaces, Combinatorica 38 (2018), 255--278.

\bibitem{BaZhou}
{\sc D. Bartoli and Y. Zhou:}
Exceptional scattered polynomials,
\emph{J. Algebra} {\bf 509} (2018), 507--534.

\bibitem{BlLa2000} 
{\sc A. Blokhuis- M. Lavrauw:}
Scattered spaces with respect to a spread in $PG(n,q)$.
Geom. Dedicata \ 81 (2000), 231--24.

\bibitem{BoPo2005}
{\sc G. Bonoli - O. Polverino:}
$\mathbb{F}_q$-linear blocking sets in $\mathrm{PG}(2,q^4)$.
Innov.\ Incidence Geom.\ 2 (2005), 35--56.

\bibitem{CsMaPo2018a}
{\sc B. Csajb\'ok - G. Marino - O. Polverino:}
A Carlitz type result for linearized polynomials.
arXiv:1804.03251

\bibitem{CsMaPo2018}
{\sc B. Csajb\'ok - G. Marino - O. Polverino:}
Classes and equivalence of linear sets in $\PG(1, q^n )$.
J.\ Combin. Theory Ser.\ A 157 (2018), 402--426.

\bibitem{CsMaPoZu2018} 
{\sc B.  Csajb\'ok - G. Marino - O. Polverino - F. Zullo:} 
A characterization of linearized polynomials with maximum kernel.
arXiv:1806.05962.

\bibitem{CsZa2016a}
{\sc B. Csajb\'ok - C. Zanella:}
On the equivalence of linear sets. Des.\ Codes Cryptogr.\ 81 (2016), 269--281.

\bibitem{CsZa2016}
{\sc B. Csajb\'ok - C. Zanella:}
On scattered linear sets of pseudoregulus type in $\PG(1,q^t)$.  
Finite Fields Appl.\ 41 (2016), 34--54.

\bibitem{CsZa2018}
{\sc B. Csajb\'ok - C. Zanella:}
Maximum scattered $\Fq$-linear sets of $\PG(1,q^4)$.
Discrete Math.\ 341 (2018), 74--80.

\bibitem{LaVdV2010}
{\sc M. Lavrauw - G. Van de Voorde:}
On linear sets on a projective line. Des.\ Codes Cryptogr. 56 (2010), 89--104.

\bibitem{LiNi1997} 
{\sc R. Lidl - H. Niederreiter:} 
Finite fields. 
Vol. 20. Cambridge university press, 1997.


\bibitem{LuPo2001} {\sc G. Lunardon - O. Polverino:} 
Blocking sets and derivable partial spreads. 
J.\ Algebraic Combin.\ 14 (2001),
    49--56.

\bibitem{LuPo2004}
{\sc G. Lunardon - O. Polverino:} Translation ovoids of orthogonal polar spaces.
Forum Math.\ 16 (2004), 663--669.


\bibitem{LuTrZh2015} 
{\sc G.  Lunardon-  R.  Trombetti  - Y.  Zhou:}
Generalized Twisted Gabidulin Codes.
J.\ Combin. Theory Ser.\ A 159 (2018), 79--106.


\bibitem{Sh2006}
{\sc J. Sheekey:} 
A new family of linear maximum rank distance codes.
Adv.\ Math.\ Commun.\ 10 (3) (2016) 475--488.

\bibitem{Sl2017} 
{\sc K. Slavov:}
An application of random plane slicing to counting $\mathbb{F}_q$-points on hypersurfaces.
arXiv:1703.05062.



\end{thebibliography}
\end{document}